\documentclass{amsart}

\usepackage{xcolor}

\usepackage{amssymb}
\usepackage{amsthm}
\usepackage{delimset}
\usepackage{enumitem}
\usepackage{thmtools}
\usepackage{thm-restate}
\usepackage{hyperref}
\usepackage{cleveref}

\setlist[enumerate,1]{label={\upshape(\roman*)}}
\setlist[enumerate,2]{label={\upshape(\alph*)}}

\theoremstyle{plain}
\newtheorem{theorem}{Theorem}[section]
\newtheorem{lemma}[theorem]{Lemma}
\newtheorem{proposition}[theorem]{Proposition}

\theoremstyle{definition}
\newtheorem{definition}[theorem]{Definition}
\newtheorem{remark}[theorem]{Remark}

\newcommand{\term}[1]{\emph{#1}} 
\newcommand{\defeq}{=_{df}} 

\newcommand{\A}{\mathcal{A}}
\newcommand{\B}{\mathcal{B}}
\newcommand{\C}{\mathcal{C}}
\newcommand{\D}{\mathcal{D}}
\newcommand{\E}{\mathcal{E}}
\newcommand{\F}{\mathcal{F}}

\newcommand{\dimdrop}{\D_{d.d.}} 
\newcommand{\Z}{\mathcal{Z}} 
\newcommand{\K}{\mathcal{K}} 

\newcommand{\Mul}{\mathcal{M}} 
\newcommand{\paschkedual}[2]{#1^d_{#2}}

\numberwithin{equation}{section}

\DeclareMathOperator{\diam}{diam}
\DeclareMathOperator{\Ped}{Ped}
\DeclareMathOperator{\Ideal}{Ideal}
\DeclareMathOperator{\her}{her}
\DeclareMathOperator{\supp}{supp}
\DeclareMathOperator{\osupp}{osupp}
\DeclareMathOperator{\dist}{dist}
\DeclareMathOperator{\trace}{tr}
\DeclareMathOperator{\ind}{ind}
\DeclareMathOperator{\id}{id}
\DeclareMathOperator{\Aut}{Aut} 
\DeclareMathOperator{\spec}{sp} 

\newcommand{\cstar}{\textrm{C}*\kern-0.15ex} 
\newcommand{\paue}{\approxeq} 
\newcommand{\nat}{\mathbb{N}} 
\newcommand{\ints}{\mathbb{Z}} 
\newcommand{\real}{\mathbb{R}} 
\newcommand{\complex}{\mathbb{C}} 
\newcommand{\mat}{M} 
\renewcommand{\k}{K} 
\newcommand{\kk}{KK} 
\newcommand{\kkz}{KK^0} 
\newcommand{\kz}{K_0} 
\newcommand{\ko}{K_1} 
\newcommand{\kkhig}{\kk_{\mathrm{Higson}}}

\newcommand{\set}{\delim\{\}}

\newcommand{\ec}{\delimpair[{[.]:}]}

\title[Paschke dual algebra]{$\ko$-injectivity of the Paschke dual algebra, \\ and uniqueness}

\author{Jireh Loreaux}

\address{Department of Mathematics and Statistics\\
Southern Illinois University Edwardsville\\
1 Hairpin Dr.\\
Edwardsville, IL\\
62026-1653\\ USA}

\email{jloreau@siue.edu}

\author{P. W. Ng}

\address{Department of Mathematics\\
University of Louisiana at Lafayette\\
217 Maxim Doucet Hall\\
P. O. Box 43568\\
Lafayette, Louisiana\\
70504--3568\\
USA}

\email{png@louisiana.edu}

\author{Arindam Sutradhar}

\address{Department of Mathematics\\
University of Louisiana at Lafayette\\
217 Maxim Doucet Hall\\
P. O. Box 43568\\
Lafayette, Louisiana\\
70504--3568\\
USA}

\email{arindam.sutradhar1@louisiana.edu}

\begin{document}

\maketitle

\begin{abstract}
  We prove that a large class of Paschke dual algebras of simple unital \cstar-algebras are $\ko$-injective.
  As a consequence, we obtain interesting $\kk$-uniqueness theorems which generalize the Brown--Douglas--Fillmore essential codimension property.
\end{abstract}

\section{Introduction}

In their seminal paper \cite[Remark~4.9]{BDF-1973-PoaCoOT}, Brown, Douglas and Fillmore
(BDF)  classified essentially normal operators using Fredholm indices.
In the course of proving functorial properties of their homology $\operatorname{Ext}(X)$, BDF introduced
the \term{essential codimension} for a pair of projections $P,Q \in B(\ell_2)$ whose
difference is compact.

Since that time, the concept of essential codimension has had numerous applications including the computation of spectral flow in semifinite von Neumann algebras \cite{BCP+-2006-Agatoeo}, as well as the explanation \cite{KL-2017-IEOT,Lor-2019-JOT} of unexpected integers arising in strong sums of projections \cite{KNZ-2009-JFA}, diagonals of projections \cite{Kad-2002-PNASUa}, and diagonals of normal operators with finite spectrum \cite{Arv-2007-PNASU,Jas-2013-JFA,BJ-2015-TAMS}.

We here present the definition of essential codimension in $B(\ell_2)$, but the translation to arbitrary semifinite factors is straightforward.

\begin{definition}
  \label{def:standard-essential-codimension}
  Given projections $P,Q \in B(\ell_2)$ with $P - Q \in \K$, the \term{essential codimension of $Q$ in $P$} is any of the following equal quantities:
  \begin{align*}
    \ec{P}{Q} &\defeq
                \begin{cases}
                  \trace(P) - \trace(Q) & \text{if}\ \trace(P) + \trace(Q) < \infty, \\[0.5em]
                  \ind(V^*W) &
                  \parbox[t]{28ex}{%
                    if $\trace(P) = \trace(Q) = \infty,$ \\
                    where $V^*V = W^*W =1,$ \\
                    $VV^* = Q, WW^* = P,$%
                  }
                \end{cases}
  \end{align*}
  where $\ind$ denotes the Fredholm index.
  We note that when $Q \le P$, $\ec{P}{Q}$ coincides with the usual codimension of $Q$ in $P$.
  Note also that $\ec{P}{Q} \in \kz(\K)$.
\end{definition}

The fundamental property introduced by
Brown--Douglas--Fillmore is encapsulated in the following theorem, whose proof extends easily to semifinite factors (see \cite{BL-2012-CJM} or \cite{KL-2017-IEOT}).

\begin{theorem}
  \label{thm:bdf-unitary-equivalence-semifinite-factors}
  If $P,Q \in B(\ell_2)$ are projections with $P-Q \in \K$, then $\ec{P}{Q} = 0$ if and only if there is a unitary $U \in 1 + \K$ conjugating $P$ to $Q$, i.e., $UPU^{*} = Q$.
\end{theorem}

It has been recognized that essential codimension can
be realized as an element of $\kkz(\complex,\K)$ (or $\kkz(\complex,\complex)$) where one identifies $\ec{P}{Q}$ with the equivalence class $[\phi,\psi]$ where $\phi,\psi : \complex \to B(\ell_2)$ are *-homomorphisms with $\phi(1) = P$ and $\psi(1) = Q$.
This leads to a natural generalization of essential codimension to the setting $\kkz(\A,\B)$ where $\A$ is a separable nuclear \cstar-algebra and $\B$ is separable stable \cstar-algebra,
and thus, uniqueness results which generalize
\Cref{thm:bdf-unitary-equivalence-semifinite-factors}
(e.g., see \cite{Lee-2011-JFA}; see also, \cite{BL-2012-CJM},
\cite{Lin-2002-JOT}, \cite{DE-2001-KT},
\cite{LN-2020-IEOT}.)

It turns out that a sufficient condition for such generalizations
is the $\ko$-injectivity of the Paschke dual algebra $\paschkedual{\A}{\B}$ (see \Cref{thm:MainUniqueness}), which is a key ingredient in Paschke duality
$\kk^j(\A, \B) \cong \k_{j+1}(\paschkedual{\A}{\B})$ for $j = 0,1$
(e.g., see \cite{Pas-1981-PJM}, \cite{Tho-2001-PAMS},
\cite{Val-1983-PJM}).
Recall that a unital \cstar-algebra $\C$ is \term{$\ko$-injective} if the natural map from $\mathcal{U}(\C)/\mathcal{U}_0(\C)$ to $K_1(\C)$ is injective.

$\ko$-injectivity of the Paschke dual algebra is, in itself, an interesting question.  For example, in the case where $\A$ is unital, the Paschke dual algebra is properly infinite (e.g., see \Cref{lem:Jan1520216AM}), and it is an interesting open question of Blanchard,
Rohde and R\o{}rdam whether every properly infinite unital \cstar-algebra is $\ko$-injective (\cite[Question~2.9]{BRR-2008-JNG}).
Consequently, we focus attention on determining conditions on $\A$ and $\B$ which guarantee that the Paschke dual algebra is $\ko$-injective.

In this paper we continue the investigation \cite{LN-2020-IEOT} by the first and second named authors (itself in the spirit of \cite{Lee-2011-JFA,BL-2012-CJM,Lee-2013-JFA,Lee-2018-JMAA}) to obtain generalizations of \Cref{thm:bdf-unitary-equivalence-semifinite-factors} in various contexts.
One such generalization from \cite{LN-2020-IEOT} was:

\begin{theorem}[\protect{\cite[Theorem~3.5]{LN-2020-IEOT}}]
  Let $\A$ be a unital separable simple nuclear \cstar-algebra, and $\B$ a separable simple stable \cstar-algebra with a nonzero projection, strict comparison of positive elements and for which $T(\B)$ has finitely many extreme points.
  Suppose that there exists a *-embedding $\A \hookrightarrow \B$.

  Let $\phi,\psi : \A \to \Mul(\B)$ be unital trivial full extensions such that $\phi(a) - \psi(a) \in \B$ for all $a \in \A$.
  Then $[\phi,\psi] = 0$ in $\kk(\A,\B)$ if and only if $\phi,\psi$ are properly asymptotically unitarily equivalent.
  \label{thm:Jan152021Restrictive}
\end{theorem}

See \Cref{def:proper-asymptotic-unitary-equivalence} for the definition of proper asymptotic unitary equivalence.
Of course, the hypotheses here are quite restrictive, and in this paper we remove many of these restraints by proving:

\begin{restatable*}{theorem}{maintheorem}
  \label{thm:main-theorem}
  Let $\A$, $\B$ be separable simple \cstar-algebras, with $\A$ unital and nuclear and $\B$ stable and $\Z$-stable.
  Let $\phi,\psi : \A \to \Mul(\B)$ be unital trivial full extensions such that $\phi(a) - \psi(a) \in \B$ for all $a \in \A$.
  Then $[\phi,\psi] = 0$ in $\kk(\A,\B)$ if and only if $\phi,\psi$ are properly asymptotically unitarily equivalent.
\end{restatable*}

As previously mentioned, the key to establishing this result, as in \cite{LN-2020-IEOT}, is to prove that the Paschke dual algebra $\paschkedual{\A}{\B}$ --- the relative commutant in the corona algebra $\C(\B)$ of the image of $\A$ under a strongly
unital trivial absorbing extension --- is $\ko$-injective.
To this end, in \Cref{sec:paschke-duals-uniq}, we fix some notation on extension theory  and the Paschke dual algebra $\paschkedual{\A}{\B}$ (see \Cref{def:paschke-dual}).
We review and slightly improve upon the properties of the Paschke dual algebra developed in \cite{LN-2020-IEOT}, many of which are generalizations of results from Paschke's original paper which was for the case when $\B = \K$.
\Cref{prop:paschke-dual-purely-infinite} guarantees that $\paschkedual{\A}{\B}$ is purely infinite when $\A$ is unital separable simple nuclear and either $\B = \K$ or $\B$ is separable stable simple purely infinite.
We conclude \Cref{sec:paschke-duals-uniq} with \Cref{thm:k1-injectivity-generic}, which is an abstract condition  on the Paschke dual algebra $\paschkedual{\A}{\B}$ sufficient to guarantee $\ko$-injectivity;
this is the main tool used in \Cref{sec:paschke-dual-simple}.
This section serves as a short summary of Paschke dual algebras, and their general properties and consequences, and as a reference for the future.

In \Cref{sec:paschke-dual-simple}, we prove that the Paschke dual algebra $\paschkedual{\A}{\B}$ is $\ko$-injective when $\A,\B$ are separable and simple with $\A$ unital nuclear and $\B$ stable and $\Z$-stable (see \Cref{lem:PIK1Injective} and \Cref{thm:k1-injectivity-simple-nuclear-strict-comparison}).
This section is quite technical.
This leads to the uniqueness result (\Cref{thm:main-theorem}).
We also have results where the $\Z$-stability of $\B$ is replaced by the hypotheses of strict comparison and that the trace space $T(\B)$ has finitely many extreme points (\Cref{thm:Apr2720218AM} and \Cref{thm:uniqueness-finite-extreme-boundary}).
The techniques in this section share many key ideas with those from \cite{LN-2020-IEOT}, but with nontrivial modifications.

In this paper, we assume that for separable simple \cstar-algebras every quasitrace is a trace.

The reader of this paper should be familiar with many aspects of \cstar-algebra theory and $\k$-theory, especially extension theory and $\kk$-theory.
It would be helpful if the reader was also knowledgeable of the basic ideas from the modern theory of simple nuclear \cstar-algebras.
The reader should be comfortable with the concepts and computations in the following books and the references therein:
\cite{Dav-1996,Weg-1993,Bla-1998,JT-1991,Lin-2001,RLL-2000}.
As examples, here are some concepts used in this paper: absorbing extensions, Busby invariant, Brown--Douglas--Fillmore theory, essential codimension, $\kk$-groups, $\ko$-injectivity, multiplier and corona algebras, properly infinite projections, Jiang--Su stability, strict comparison, cancellation of projections, nuclearity.

\section{Paschke duals and uniqueness}
\label{sec:paschke-duals-uniq}

We begin by fixing some notation and recalling some basic facts from extension theory.
More detailed references for extension theory can be found in \cite{Weg-1993,Bla-1998}.
Given a \cstar-algebra $\B$, we let $\Mul(\B)$ denote the multiplier algebra of $\B$, and $\C(\B) \defeq \Mul(\B)/\B$ denote the corona algebra of $\B$.
Recall that given an extension of \cstar-algebras
\begin{equation*}
  0 \to \B \to \mathcal{E} \to \A \to 0,
\end{equation*}
one associates the \term{Busby invariant}, which is a *-homomorphism $\phi : \A \to \C(\B)$.
Conversely, given such a *-homomorphism $\phi : \A \to \C(\B)$, one can obtain an extension of $\B$ by $\A$ whose Busby invariant is $\phi$.
It is well-known that the extension corresponding to a given Busby invariant is unique up to \term{strong isomorphism} (in the terminology of Blackadar; see \cite[15.1--15.4]{Bla-1998}).
Our results are all invariant under strong isomorphism, and therefore we will simply refer to $\phi : \A \to \C(\B)$ as an extension.

A \term{trivial extension} is one for which the short exact sequence is split exact, or equivalently, the Busby invariant factors through the quotient map $\pi : \Mul(\B) \to \C(\B)$ via a *-homomorphism $\phi_0 : \A \to \Mul(\B)$ so that $\phi = \pi \circ \phi_0$.
If $\phi : \A \to \C(\B)$ is a trivial extension, then by a slight abuse of terminology we also refer to the map $\phi_0$ as a trivial extension.
When $\A$ is a unital \cstar-algebra and $\phi$ is a unital *-homomorphism, we say that $\phi$ is a \term{unital extension}.
If, in addition, $\phi$ is trivial and some lift $\phi_0$ is unital, $\phi$ is said to be a \term{strongly unital trivial extension}.
When $\phi$ is injective, then $\phi$ is called an \term{essential extension}.  This corresponds to $\B$ being an essential ideal of $\E$.     Two extensions $\phi, \psi : \A \rightarrow \C(\B)$ are said to be \term{unitarily equivalent} (and denoted $\phi \sim \psi$) if there exists a unitary
$U \in \Mul(\B)$ such that
\begin{equation*}
  \pi(U) \phi(\cdot) \pi(U)^* = \psi(\cdot).
\end{equation*}
(In the literature, unitary equivalence is sometimes called \term{strong unitary equivalence}.)

Suppose that $\phi_0, \psi_0 : \A \rightarrow \Mul(\B)$ are *-homomorphisms.  We will sometimes write $\phi_0 \sim \psi_0$
to mean $\pi \circ \phi_0 \sim \pi \circ \psi_0$, i.e., that the corresponding extensions are unitarily equivalent.

When $\B$ is stable, there are isometries $V,W \in \Mul(\B)$ for which $VV^{*} + WW^{*} = 1$.
Given extensions $\phi,\psi : \A \to \C(\B)$, their \term{BDF sum} is the extension
\begin{equation*}
  (\phi \oplus \psi)(\cdot) \defeq \pi(V) \phi(\cdot) \pi(V^{*}) + \pi(W) \psi(\cdot) \pi(W^{*}).
\end{equation*}
It is not hard to see that the BDF sum is well-defined up to unitary equivalence (i.e., up to $\sim$).
An extension $\phi : \A \to \C(\B)$ is \term{absorbing} (respectively, \term{unitally absorbing}) if
\begin{equation*}
  \phi \sim \phi \oplus \psi
\end{equation*}
for any (respectively, \term{strongly unital}) trivial extension $\psi$.
Note that since any absorbing extension must absorb the zero *-homomorphism, it is necessarily nonunital.  The convention (which we will follow) is that if $\phi$ is a unital extension then when we say that $\phi$ is \term{absorbing}, we mean that $\phi$ is unitally absorbing.

In \cite{Tho-2001-PAMS}, Thomsen establishes the existence of an absorbing trivial extension of $\B$ by $\A$, when $\A,\B$ are separable with $\B$ stable, and provides several equivalent characterizations of absorbing trivial extensions.
Important precursors to Thomsen's work, which are themselves of interest and give significantly more information, are contained in \cite{Voi-1976-RRMPA,Kas-1980-JOT,Lin-2002-JOT,EK-2001-PJM}.

In \cite{Pas-1981-PJM}, Paschke introduced, for a separable unital \cstar-algebra $\A$ and a unital trivial essential extension $\phi : \A \to B(\ell_2(\nat))$ (which is necessarily absorbing by Voiculescu's theorem \cite{Voi-1976-RRMPA}) the subalgebra $(\pi \circ \phi (\A))'$ of the Calkin algebra and proved, in the language of $\kk$-theory, $\k_j((\pi \circ \phi (\A))') \cong \kk^{j+1}(\A,\complex)$.
Soon after in \cite{Val-1983-PJM}, Valette extended these ideas and techniques to a pair of algebras $\A,\B$.
In particular, Valette proved \cite[Proposition~3]{Val-1983-PJM} that if $\A$ is a separable unital nuclear \cstar-algebra and $\B$ is a stable $\sigma$-unital \cstar-algebra, then $\k_j((\pi \circ \phi (\A))') \cong \kk^{j+1}(\A,\B)$, where $\phi : \A \to \Mul(\B)$ is a unital absorbing trivial extension.   This is the so-called \term{Paschke duality}, and it has been generalized to more general algebras (e.g., see \cite{Tho-2001-PAMS}), but in this paper we focus on the case when $\A$ is unital and nuclear.
It is for this reason that this algebra $(\pi \circ \phi (\A))'$ gets its name, the \term{Paschke dual algebra}.

\begin{definition}
  \label{def:paschke-dual}
  Let $\A,\B$ be separable \cstar-algebras with $\A$ unital and $\B$ stable.
  Let $\phi : \A \to \Mul(\B)$ be a unital absorbing trivial extension.
  The \term{Paschke dual algebra} $\paschkedual{\A}{\B}$ is the relative commutant in the corona algebra $\C(\B)$ of the image of $\A$ under $\pi \circ \phi$, i.e., $\paschkedual{\A}{\B} \defeq \big(\pi \circ \phi (\A)\big)'$.
\end{definition}

The Paschke dual algebra is independent, up to *-isomorphism, of the choice of the absorbing (strongly) unital trivial extension $\phi$.
Indeed, if $\phi, \psi : \A \to \Mul(\B)$ are unital absorbing trivial extensions, then $\phi \sim \phi \oplus \psi \sim \psi$, and so there is a unitary $U \in \Mul(\B)$ such that $\pi(U)(\pi \circ \phi(a))\pi(U^{*}) = \pi \circ \psi(a)$ for all $a \in \A$.
Then this unitary also conjugates the relative commutants of $\pi \circ \phi(\A)$ and $\pi \circ \psi(\A)$.

In \cite{LN-2020-IEOT}, we proved several results about $\paschkedual{\A}{\B}$, generalizing those studied by Paschke for $\paschkedual{\A}{\K}$, which we summarize in the next lemma.

\begin{lemma}[\protect{\cite[Lemma~2.2]{LN-2020-IEOT}}]
  \label{lem:paschke-dual-properly-infinite}
  Let $\A$ be a unital separable nuclear \cstar-algebra, and let $\B$ be a separable stable \cstar-algebra.
  Then we have the following:
  \begin{enumerate}
  \item The Paschke dual algebra $\paschkedual{\A}{\B}$ is properly infinite.  In fact, $1 \oplus 0 \sim 1 \oplus 1$ in $\mat_2 \otimes \paschkedual{\A}{\B}$, i.e., $\paschkedual{\A}{\B}$ contains a unital copy of the Cuntz algebra $\mathcal{O}_2$.
  \item Every element of $\kz(\paschkedual{\A}{\B})$ is represented by a full properly infinite projection in $\paschkedual{\A}{\B}$.
  \end{enumerate}
  \label{lem:Jan1520216AM}
\end{lemma}

In \cite{LN-2020-IEOT}, we also provided the following double commutant theorem for the Paschke dual algebra, which is akin to \cite[Theorem~1]{Ng-2018-NYJM}, and shows that the algebra $\paschkedual{\A}{\B}$ is dual in yet another way.   This generalizes a remark of Valette \cite{Val-1983-PJM}.

\begin{theorem}[\protect{\cite[Theorem~2.10]{LN-2020-IEOT}}]
  \label{thm:standard-position-relative-commutants}
  Let $\A$ be a separable simple unital nuclear \cstar-algebra, and let $\B$ be a separable stable simple \cstar-algebra.
  For any unital trivial absorbing extension $\phi : \A \to \Mul(\B)$,
  \begin{equation*}
    \big(\pi \circ \phi(\A)\big)'' = \pi \circ \phi (\A).
  \end{equation*}
  Equivalently, if we identify $\A$ with its image $\pi \circ \phi (\A)$ (since $\pi \circ \phi$ is injective), and $\paschkedual{\A}{\B} \defeq \big(\pi \circ \phi(\A)\big)'$, then
  \begin{equation*}
    \A' = \paschkedual{\A}{\B} \quad\text{and}\quad (\paschkedual{\A}{\B})' = \A.
  \end{equation*}
\end{theorem}

Towards generalizations of \Cref{thm:bdf-unitary-equivalence-semifinite-factors},
we remind the reader of the following definition due to \cite{DE-2001-KT}.

\begin{definition}
  \label{def:proper-asymptotic-unitary-equivalence}
  Let $\A, \B$ be separable \cstar-algebras, and let
  $\phi, \psi : \A \rightarrow \Mul(\B)$ be *-homomorphisms.

  $\phi$ and $\psi$ are said to be \term{properly asymptotically unitarily equivalent}
  ($\phi \paue \psi$) if there exists a norm continuous path
  $\{ u_t \}_{t \in [0, \infty)}$ of unitaries in $\complex 1 + \B$ such that
  \begin{equation*}
    u_t \phi(a) u_t^* - \psi(a) \in \B, \makebox{   for all }t
  \end{equation*}
  and
  \begin{equation*}
    \norm{ u_t \phi(a) u_t^* - \psi(a) } \rightarrow 0 \makebox{   as  } t \rightarrow
    \infty,
  \end{equation*}
  for all $a \in \A$.
\end{definition}

As preparation for the next result, we remark that the corona algebra of a separable stable \cstar-algebra is $\ko$-injective (see \Cref{lem:C(B)Injective}).

The next result says that $\ko$-injectivity of the Paschke dual $\paschkedual{\A}{\B}$ is sufficient to guarantee interesting uniqueness theorems which generalize \Cref{thm:bdf-unitary-equivalence-semifinite-factors}.
The proof is already contained in previous works, though often implicitly (\cite{DE-2001-KT,Lee-2011-JFA,LN-2020-IEOT, Lin-2002-JOT}).  As part of the purpose of this section is to provide an accessible reference,
for the convenience of the reader and to help clean up the literature, we explicitly provide the statement and argument.

\begin{theorem}
  Let $\A$ be a separable nuclear \cstar-algebra, and let $\B$ be a separable stable \cstar-algebra.
  Suppose that either $\A$ is unital and $\paschkedual{\A}{\B}$ is $\ko$-injective, or $\A$ is nonunital and $\paschkedual{(\widetilde{\A})}{\B}$ is $\ko$-injective.
  Let $\phi, \psi : \A \rightarrow \Mul(\B)$ be two absorbing trivial extensions with $\phi(a) - \psi(a) \in \B$ for all $a \in \A$ such that either both $\phi$ and $\psi$ are unital or both $\pi \circ \phi$ and $\pi \circ \psi$ are nonunital.

  Then
  $[\phi, \psi] = 0$ in $\kk(\A, \B)$ if and only if $\phi \paue \psi$.
  \label{thm:MainUniqueness}
\end{theorem}

\begin{proof}

  The ``if'' direction follows directly from Lemma~3.3 of \cite{DE-2001-KT}.

  We now prove the ``only if'' direction.

  Let $\A^+$ denote the unitization of $\A$ if $\A$ is nonunital, and
  $\A \oplus \complex$ if $\A$ is unital.
  Now if $\phi : \A \rightarrow \Mul(\B)$ is a
  nonunital absorbing extension (so $\pi \circ \phi$ is nonunital),
  then $\phi(\A)^{\perp}$ contains a projection which is Murray--von Neumann equivalent
  to $1_{\Mul(\B)}$, and by
  \cite{EK-2001-PJM},
  the map $\phi^+ :
  \A^+ \rightarrow \Mul(\B)$ given by
  $\phi^+ |_{\A} = \phi$ and
  $\phi^+(1) = 1$ is a unital absorbing trivial extension (i.e., $\pi \circ \phi^+$ is unitally absorbing).
  The same holds for $\psi$ and $\psi^+$.
  Moreover, $(\phi^+,\psi^+)$ is a generalized homomorphism.
  Additionally, $[\phi^+,\psi^+] = 0$ because a homotopy of generalized homomorphisms $(\phi_s,\psi_s)$ between $(\phi,\psi)$ and $(0,0)$ lifts to a homotopy $(\phi^+_s,\psi^+_s)$, and hence $[\phi^+,\psi^+] = [0^+,0^+] = 0$.
  Thus, we may assume that
  $\A$ is unital and $\phi$ and $\psi$ are unital *-monomorphisms.

  As before, we may identify the Paschke dual algebra as $\paschkedual{\A}{\B} = (\pi \circ \phi(\A))' \subseteq \C(\B)$.

  By \cite[Lemma~3.3]{LN-2020-IEOT} there exists a norm continuous path  $\{ u_t \}_{t \in [0, \infty)}$ of unitaries in $\Mul(\B)$ such that
  \begin{equation*}
    u_t \phi(a) u_t^* - \psi(a) \in \B
  \end{equation*}
  for all $t$ and for all $a \in \A$, and
  \begin{equation*}
    \norm{u_t \phi(a) u_t^* - \psi(a)} \rightarrow 0
  \end{equation*}
  as $t \rightarrow \infty$, for all $a \in \A$.

  It is trivial to see that this implies that
  \begin{equation*}
    [\phi, u_0\phi u_0^*] = [ \phi, \psi ] = 0,
  \end{equation*}
  and that $\pi(u_t) \in (\pi \circ \phi(\A))' = \paschkedual{\A}{\B}$ for all $t$.

  It is well-known that we have
  a group isomorphism $\kk(\A, \B) \rightarrow \kkhig(\A, \B) : [\phi, \psi]
  \rightarrow [\phi, \psi, 1]$.
  (Here, $\kkhig$ is the version of $\kk$-theory presented in \cite{Hig-1987-PJM}
  Section 2.)
  Hence, $[\phi, u_0 \phi u_0^*, 1] = 0$ in $\kkhig(\A, \B)$.
  Hence, by \cite[Lemma~2.3]{Hig-1987-PJM}, $[\phi, \phi, u_0^*] = 0$
  in $\kkhig(\A, \B)$.

  By Thomsen's Paschke duality theorem (\cite{Tho-2001-PAMS} Theorem 3.2),
  there is a group isomorphism $\ko(\paschkedual{\A}{\B}) \rightarrow \kkhig(\A, \B)$
  which sends $[\pi(u_0)]$ to $[\phi, \phi, u_0^*]$.
  Hence, $[\pi(u_0)] = 0$ in $\ko(\paschkedual{\A}{\B})$.
  Since $\paschkedual{\A}{\B}$ is $\ko$-injective, $\pi(u_0) \sim_h 1$ in $\paschkedual{\A}{\B} =
  (\pi \circ \phi(\A))'$.
  Hence, there exists a unitary $v \in \complex 1 + \B$ such that
  $v^* u_0 \sim_h 1$ in $\pi^{-1}(\paschkedual{\A}{\B})$.

  Hence, modifying an initial segment of $\{ v^*u_t \}_{t \in [0, \infty)}$
  if necessary, we may assume that $\{ v^*u_t \}_{t \in [0, \infty)}$ is
  a norm continuous path of unitaries in $\pi^{-1}(\paschkedual{\A}{\B})$ such that
  $v^*u_0 = 1$.

  Now for all $t \in [0,\infty)$, let $\alpha_t \in \Aut(\phi(\A) + \B)$ be given
  by $\alpha_t(x) \defeq v^* u_t x u_t^* v$ for all $x \in \phi(\A) + \B$.
  Thus, $\{ \alpha_t \}_{t \in [0, \infty)}$ is a uniformly continuous path of
  automorphisms of $\phi(\A) + \B$ such that $\alpha_0 = \id$.
  Hence, by \cite[Proposition~2.15]{DE-2001-KT} (see also
  \cite[Theorems~3.2 and 3.4]{Lin-2002-JOT}),
  there exist a continuous path $\{ v_t \}_{t\in [0, \infty)}$ of unitaries
  in $\phi(\A) + \B$ such that
  $v_0 = 1$ and
  $\norm{ v_t x v_t^* - v^* u_t x u_t^* v } \rightarrow 0$ as $t \rightarrow \infty$
  for all $x \in \phi(\A) + \B$.
  Thus, $\norm{ v v_t x v_t^* v^* - u_t x u_t^* } \rightarrow 0$ as
  $t \rightarrow \infty$
  for all $x \in \phi(\A) + \B$.

  We now proceed as in the last part of the proof of \cite[Proposition 3.6. Step 1]{DE-2001-KT} (see
  also the proof of \cite[Theorem~3.4]{Lin-2002-JOT}).
  For all $t \in [0, \infty)$, let $a_t \in \A$ and $b_t \in \B$ such that
  $v v_t = \phi(a_t) + b_t$.
  Since $\pi \circ \phi$ is injective, we have that for all $t$,
  $a_t$ is a unitary in $\A$, and hence, $\phi(a_t)$ is a unitary in $\phi(\A) + \B$.
  Note also that since $\pi \circ \phi = \pi \circ \psi$ and both maps are
  injective, $\norm{ a_t a a_t^* - a } \rightarrow 0$ as $t \rightarrow \infty$ for
  all $a \in \A$.
  For all $t$, let $w_t \defeq v v_t \phi(a_t)^* \in 1 + \B$.
  Then $\{ w_t \}_{t \in [0,1)}$ is a norm continuous path of unitaries in
  $\complex 1 + \B$, and for all $a \in \A$,
  \begin{align*}
    \norm{w_t \phi(a) w_t^* - \psi(a)}
    &\leq \norm{ w_t \phi(a) w_t^* - v v_t \phi(a) v_t^* v^* } \\
    &\qquad+ \norm{ v v_t \phi(a) v_t^* v^* - u_t \phi(a) u_t^* } \\
    &\qquad+ \norm{ u_t \phi(a) u_t^* - \psi(a) } \\
    &= \norm{ v v_t\phi( a_t^* a a_t - a) v_t^* v^* } \\
    &\qquad+ \norm{ v v_t \phi(a) v_t^* v^* - u_t \phi(a) u_t^* } \\
    &\qquad+ \norm{ u_t \phi(a) u_t^* - \psi(a) } \\
    &\rightarrow 0. \qedhere
  \end{align*}

\end{proof}

\begin{lemma}[\protect{\cite[Lemma~2.4]{LN-2020-IEOT}}]
  Let $\A$ be a unital separable nuclear \cstar-algebra and $\B$ be a separable
  stable \cstar-algebra such that either $\B \cong \K$ or $\B$ is
  simple purely infinite.

  Then $\paschkedual{\A}{\B}$ is $\ko$-injective.
  \label{lem:PIK1Injective}
\end{lemma}

\begin{theorem}
  Let $\A$, $\B$ be separable \cstar-algebras such that $\A$ is nuclear and
  either $\B \cong \K$ or $\B$ is stable simple purely infinite.
  Let $\phi, \psi : \A \rightarrow \Mul(\B)$ be essential trivial extensions
  with $\phi(a) - \psi(a) \in \B$ for all $a \in \A$  such that
  either both $\phi$ and $\psi$ are unital or both $\pi \circ \phi$ and
  $\pi \circ \psi$ are nonunital.

  Then $[\phi, \psi] = 0$ in $\kk(\A, \B)$ if and only if $\phi \paue \psi$.
\end{theorem}

\begin{proof}
  This follows immediately from \Cref{thm:MainUniqueness} and
  \Cref{lem:PIK1Injective}.

  Of course, we here are using that we are in the ``nicest'' setting for
  extension theory:  By our hypotheses on $\A$ and $\B$, every essential
  extension of $\B$ by $\A$ is absorbing in the appropriate sense (see, for example, \cite[Proposition~2.5]{GN-2019-AM}).
\end{proof}

Thus, from the point of view of simple stable canonical ideals with appropriate
regularity properties,
what remains is the case where the canonical ideal is stably finite.  In
\cite{LN-2020-IEOT}, we had a partial result with restrictive conditions
(see the present paper \Cref{thm:Jan152021Restrictive}).  Part of
the goal of the present paper is to remove many of these restrictive conditions
(see the present paper \Cref{thm:main-theorem}).

The next result partially answers a question that Professor Huaxin Lin
asked the second author.  The argument actually comes from Lin's
paper \cite[Proposition~2.6]{Lin-2005-CJM}.

\begin{proposition}
  \label{prop:paschke-dual-purely-infinite}
  Let $\A$ be a unital separable simple nuclear \cstar-algebra, and let
  $\B$ be a separable stable \cstar-algebra such that
  either
  \begin{equation*}
    \B \cong \K \makebox{  or  } \B \makebox{  is simple purely infinite.}
  \end{equation*}

  Let $\sigma : \A \rightarrow \Mul(\B)$ be a unital trivial essential
  extension.

  Then $(\pi \circ \sigma (\A))'$  is simple purely infinite.
  As a consequence, $\paschkedual{\A}{\B}$ is simple purely infinite.
\end{proposition}

\begin{proof}
  Note that $\sigma$ is absorbing (e.g., see
  \cite[Proposition~2.5]{GN-2019-AM}).

  By \Cref{lem:paschke-dual-properly-infinite}, $\paschkedual{\A}{\B}$ contains a unital copy of $\mathcal{O}_2$, so it cannot be isomorphic to $\complex$.
  Let $c \in (\pi \circ \sigma(\A))'_+ - \{ 0 \}$ be arbitrary.
  We want to find $s \in (\pi \circ \sigma(\A))'$ so that
  $s c s^* = 1$.

  We may assume that $\norm{c} = 1$.

  Let
  \begin{equation*}
    X \defeq \spec(c).
  \end{equation*}

  Let
  \begin{equation*}
    \phi, \psi : C(X) \otimes \A \rightarrow \Mul(\B)/\B
  \end{equation*} be *-homomorphisms
  that are given as follows (using the universal property of the maximal tensor product):

  \begin{equation*}
    \phi : f \otimes a \mapsto f(c) (\pi \circ \sigma(a))
  \end{equation*}
  and
  \begin{equation*}
    \psi : f \otimes a \mapsto f(1) (\pi \circ \sigma(a)).
  \end{equation*}

  Both $\phi$ and $\psi$ are unital extensions.  Since $\A$ is simple unital, it is a short exercise to see that $\phi$ is essential.  Moreover,
  one can check that $\psi$ is a strongly unital trivial extension.

  Hence, since either $\B \cong \K$ or $\B$ is simple purely infinite,
  we must have that $\phi$ is absorbing (from \cite[Proposition~2.5]{GN-2019-AM}), so
  \begin{equation*}
    \phi \sim \phi \oplus \psi
  \end{equation*}
  where $\oplus$ is the BDF sum, and $\sim$ is unitary equivalence with unitary coming from $\Mul(\B)$.

  So let $W \in \mat_2 \otimes \Mul(\B)$ be such that
  \begin{equation*}
    W^* W = 1 \oplus 1,
  \end{equation*}
  \begin{equation*}
    W W^* = 1 \oplus 0,
  \end{equation*}
  and letting $w \defeq \pi(W)$
  \begin{equation*}
    w^*(\phi(\cdot) \oplus 0)w = \psi(\cdot) \oplus \phi(\cdot).
  \end{equation*}

  Then
  \begin{align*}
    w^*(c \oplus 0)w
    &= w^*(\phi(\id_X \otimes 1) \oplus 0)w\\
    &= \psi(\id_X \otimes 1) \oplus \phi(\id_X \otimes 1)\\
    &= 1 \oplus c. 
  \end{align*}

  So
  \begin{equation*}
    \left[\begin{array}{cc} 1 & 0 \\ 0 & 0 \end{array}\right]
    w^* (c \oplus 0) w \left[\begin{array}{cc} 1 & 0 \\ 0 & 0 \end{array} \right]
    = 1 \oplus 0.
  \end{equation*}

  Identifying $\Mul(\B)$ with $e_{1,1} \otimes \Mul(\B)$, we may view
  $w \left[ \begin{array}{cc} 1 & 0 \\ 0 & 0 \end{array} \right]$ as an
  element of $\C(\B)$.  Hence, from the above computation, to finish the
  proof, it suffices to prove that
  \begin{equation*}
    w \left[\begin{array}{cc} 1 & 0 \\ 0 & 0 \end{array} \right]
    \in (\pi \circ \sigma(\A))'.
  \end{equation*}

  But for all $a \in \A$,
  \begin{align*}
    w^* ( \pi \circ \sigma(a) \oplus 0) w
    &=  w^* (\phi(1 \otimes a) \oplus 0)w \\
    &= \psi(1 \otimes a) \oplus \phi(1 \otimes a)\\
    &= \pi \circ \sigma(a) \oplus \pi \circ \sigma(a).
  \end{align*}

  So
  \begin{equation*}
    (\pi \circ \sigma(a) \oplus 0) w =
    w ( \pi \circ \sigma(a) \oplus \pi \circ \sigma(a))
  \end{equation*}
  for all $a \in \A$.

  Since $ww^* = 1 \oplus 0$,
  \begin{equation*}
    (\pi \circ \sigma(a) \oplus \pi \circ \sigma(a))w =
    w (\pi \circ \sigma(a) \oplus \pi \circ \sigma(a))
  \end{equation*}
  for all $a \in \A$.

  Hence,
  \begin{equation*}
    w \in \mat_2 \otimes (\pi \circ \sigma(\A))'.
  \end{equation*}

  Hence,
  \begin{equation*}
    w \left[ \begin{array}{cc} 1 & 0 \\ 0 & 0 \end{array} \right]
    \in (\pi \circ \sigma(\A))'
  \end{equation*}
  as required.
\end{proof}

Our next result establishes a generic sufficient criterion for establishing the $\ko$-injectivity of the Paschke dual algebra.
It is the key technique used in the next section.

Let $\C \subseteq \D$ be an inclusion of \cstar-algebras. Recall that
$\C$ is said to be \term{strongly full} in $\D$ if
every nonzero element of $\C$ is full in $\D$, i.e., for all $y \in \C - \{ 0 \}$, $\Ideal_{\D}(y) = \D$.
Recall also that an element $x \in \D$ is said to be \term{strongly full} in $\D$ if $C^*(x)$ is strongly full in $\D$.
Finally, an extension $\phi : \A \to \C(\B)$ is said to be \term{full} if $\phi$ is essential and $\phi(\A)$ is strongly full in $\C(\B)$.

We remark on a key fact we use in the proof: if $\A$ is a unital separable nuclear \cstar-algebra and $\B$ is a separable stable with the corona factorization property (CFP), then every unital full extension of $\B$ by $\A$ is absorbing \cite{KN-2006-HJM}.
Many separable simple stable \cstar-algebras have the corona factorization property, including all such \cstar-algebras with strict comparison of positive elements, and also all such \cstar-algebras which are purely infinite (see \cite{OPR-2011-TAMS,OPR-2012-IMRN}).

\begin{theorem}
  \label{thm:k1-injectivity-generic}
  Suppose that $\A,\B$ are separable \cstar-algebras with $\A$ unital, simple and nuclear, and $\B$ stable and has the corona factorization property.

  Suppose that $\phi : \A \to \Mul(\B)$ is a unital trivial absorbing extension and realize $\paschkedual{\A}{\B}$ as $(\pi \circ \phi (\A))'$.
  If every unitary in the Paschke dual algebra $\paschkedual{\A}{\B}$ is homotopic in $\paschkedual{\A}{\B}$ to a unitary which is strongly full in $\C(\B)$, then $\paschkedual{\A}{\B}$ is $\ko$-injective.
  Moreover, for each $n \in \nat$, the map
  \begin{equation*}
    U(\mat_n \otimes \paschkedual{\A}{\B})/U(\mat_n \otimes \paschkedual{\A}{\B})_0 \to U(\mat_{2n} \otimes \paschkedual{\A}{\B})/U(\mat_{2n} \otimes \paschkedual{\A}{\B})_0
  \end{equation*}
  given by
  \begin{equation*} [u] \mapsto [u \oplus 1] \end{equation*}
  is injective.
\end{theorem}

\begin{proof}
  Clearly, $\ko$-injectivity will follow from the more specific statement.
  By \Cref{lem:paschke-dual-properly-infinite}, $\paschkedual{\A}{\B}$ contains a unital copy of $\mathcal{O}_2$, and so $1 \oplus 1 \sim 1$.
  Therefore, for all $n \in \nat$, $\paschkedual{\A}{\B} \cong \mat_n \otimes \paschkedual{\A}{\B}$.
  Hence it suffices to establish injectivity for the map in the $n = 1$ case: \begin{equation*}
    U(\paschkedual{\A}{\B})/U(\paschkedual{\A}{\B})_0 \to U(\mat_2 \otimes \paschkedual{\A}{\B})/U(\mat_2 \otimes \paschkedual{\A}{\B})_0.
  \end{equation*}

  Let $u \in \paschkedual{\A}{\B}$ be a unitary for which
  \begin{equation*}
    u \oplus 1 \sim_h 1 \oplus 1
  \end{equation*} in $\mat_2 \otimes \paschkedual{\A}{\B}$.

  By hypothesis, $u$ is homotopic, in $\paschkedual{\A}{\B}$, to a unitary which is strongly full in $\C(\B)$.
  Therefore, we may assume, without loss of generality, that $u$ is strongly full in $\C(\B)$.
  By \cite[Lemma~2.6]{LN-2020-IEOT}, $C^{*}(\pi \circ \phi(\A),u)$ is strongly full in $\C(\B)$.
  Hence, the inclusion map
  \begin{equation*}
    \i : C^{*}(\pi \circ \phi(\A),u) \hookrightarrow \C(\B)
  \end{equation*}
  is a full extension.  Since $\B$ has the corona factorization property, and since $C^*(\pi \circ \phi(\A), u)$ is nuclear (being a quotient of $C(\mathbb{S}^1) \otimes \A$), $\i$ is a unital absorbing extension.

  Let $\sigma : C^{*}(\pi \circ \phi(\A),u) \to \Mul(\B)$ be a unital trivial absorbing extension.
  Since the restriction $\sigma |_{\pi \circ \phi(\A)}$ is also a unital trivial absorbing extension, conjugating
  $\sigma$ by a unitary if necessary, we may assume that the restriction of $\pi \circ \sigma$ to $\pi \circ \phi(\A)$ is the identity
  map.

  By \cite[Lemma 2.3]{LN-2020-IEOT},
  \begin{equation*}
    \pi \circ \sigma(u) \sim_h 1
  \end{equation*}
  in $\paschkedual{\A}{\B}$.

  Since $\i$ is absorbing,
  \begin{equation*}
    \i \oplus (\pi \circ \sigma) \sim \i.
  \end{equation*}

  Consequently, there is an isometry $\tilde{v} \in \mat_2 \otimes \Mul(\B)$ such that $\tilde{v}^{*} \tilde{v} = 1 \oplus 1$, $\tilde{v} \tilde{v}^{*} = 1 \oplus 0$, and for $v \defeq \pi(\tilde{v})$,
  \begin{equation*}
    v (\i \oplus (\pi \circ \sigma)) v^{*} = \i \oplus 0.
  \end{equation*}
  Therefore,
  \begin{equation*}
    v(a \oplus a)v^{*} = a \oplus 0
  \end{equation*} for all $a \in \pi \circ \phi(\A)$, and also
  \begin{equation*}
    v(u \oplus (\pi \circ \sigma(u)))v^{*} = u \oplus 0.
  \end{equation*}
  Thus, since $vv^{*} = 1 \oplus 0$,
  \begin{equation*}
    v(a \oplus a) = (a \oplus 0)v = (a \oplus a) v.
  \end{equation*} for all $a \in \pi \circ \phi(\A)$, and therefore,
  \begin{equation*}
    v \in \mat_2 \otimes \paschkedual{\A}{\B}.
  \end{equation*}

  Moreover, we also have
  \begin{equation*}
    u \oplus (\pi \circ \sigma (u)) \sim_h u \oplus 1 \sim_h 1 \oplus 1
  \end{equation*}
  within $\mat_2 \otimes \paschkedual{\A}{\B}$.
  Conjugating this continuous path of unitaries by $v$, we obtain
  \begin{equation*}
    u \sim_h 1
  \end{equation*} within $\paschkedual{\A}{\B}$.
\end{proof}

\section{The Paschke dual for simple nuclear \cstar-algebras}
\label{sec:paschke-dual-simple}

In this section, we prove our main result that the Paschke dual algebra $\paschkedual{\A}{\B}$ is $K_1$ injective, when $\A$, $\B$ are separable simple \cstar-algebras with $\A$ unital and nuclear, and $\B$ stable and $\Z$-stable.
Since this section is very long and technical, we begin by providing a rough summary of the main ideas of the argument.
We advise the reader to read this summary first and to use it as a guide while going through the technical details.

The main strategy of our argument is to use \Cref{thm:k1-injectivity-generic}.
For this, for some unital trivial absorbing extension $\phi : \A \to \Mul(\B)$, it suffices to show that given any $U \in \paschkedual{\A}{\B} := (\pi \circ \phi (\A))'$ there is a unitary $V$ homotopic to $1$ in $\paschkedual{\A}{\B}$ such that $VU$ is strongly full in $\C(\B)$.

Firstly, we summarize what we did in the paper \cite{LN-2020-IEOT} because the argument is simpler, and understanding it should help the reader better comprehend our modifications necessary for the current argument.
In \cite{LN-2020-IEOT}, we assumed that $\A$, $\B$ were separable simple \cstar-algebras with $\A$ unital and nuclear, $\B$ stable and having strict comparison, $T(\B)$ having finitely many extreme points, and for which there exists an embedding $\rho : \A \hookrightarrow \B$.
Let $p \defeq \rho(1_{\A}) \in \B$.
Then there exists a sequence $\{ p_k \}_{k=0}^{\infty}$  of pairwise orthogonal projections in $\B$ such that
\begin{equation*}
  p_0 = p \text{ and } p_j \sim p_k \text{ for all } j,k,
\end{equation*}
and
\begin{equation*}
  \sum_{k=0}^{\infty} p_k = 1_{\Mul(\B)},
\end{equation*}
where the sum converges strictly in $\Mul(\B)$.
For all $j \geq 0$, let $v_j \in \B$ be a partial isometry in $\B$ for which $v_j^*v_j = p_0$ and $v_j v_j^* = p_j$.
Let $\phi : \A \rightarrow \Mul(\B)$ be the unital *-monomorphism given by
\begin{equation*}
  \phi(a) \defeq \sum_{j=0}^{\infty} v_j \rho(a) v_j^*.
\end{equation*}
Then $\phi$ is a (unitally) absorbing extension.
(This is the Lin extension.)
We realize the Paschke dual algebra as $\paschkedual{\A}{\B} = (\pi \circ \phi(\A))'$.
Following the strategy of \Cref{thm:k1-injectivity-generic}, given a unitary $ U\in \paschkedual{\A}{\B}$, we want to find a unitary $V \in \paschkedual{\A}{\B}$, which is homotopic to $1$ in $\paschkedual{\A}{\B}$, such that $VU$ is strongly full in $\C(\B)$.
Notice that by our concrete choice of $\phi$, for every diagonal unitary of the form $Z \defeq \sum_{j=0}^{\infty} \alpha_j p_j$ ($\alpha_j \in S^1$), $V \defeq \pi(Z) \in \paschkedual{\A}{\B}$ and $\pi(Z)$ is a unitary which is homotopic to $1_{\paschkedual{\A}{\B}}$ in $\paschkedual{\A}{\B}$.
Using that $\B$ has strict comparison and $T(\B)$ has finitely many extreme points, we showed that we could choose $\alpha_j \in S^1$ so that $\pi(Z) U$ is strongly full in $\C(\B)$.
This complicated argument can be simply (and misleadingly) described as choosing $\{ \alpha_j \}$ so that we could ``distribute mass over the spectrum of $\pi(Z)U$.''
Since $U$ was arbitrary, by applying \Cref{thm:k1-injectivity-generic}, we have that $\paschkedual{\A}{\B}$ is $K_1$ injective.

We emphasize that the assumptions made in \cite{LN-2020-IEOT} were highly restrictive.
That $\A$ is assumed to be unitally embedded into $\B$ is a great constraint, since, first of all this implies that $\B$ has a nonzero projection, and in the theory of simple \cstar-algebras, there are many interesting regular \cstar-algebras $\B$ which have no projection other than zero.
More seriously, this rules out whole categories of examples --- for instance, the case where $\A$ is purely infinite and $\B$ is stably finite.
The requirement that $T(\B)$ has finite extreme boundary is another enormous constraint, since we know that every metrizable Choquet simplex can be realized as $T(\B)$.
Note that this last requirement was used for our method for getting strong fullness of $\pi(Z) U$.

In the present paper, we remove the aforementioned restrictions with the additional assumption of Jiang--Su or $\Z$-stability of $\B$.
We mention here (and in a few paragraphs) that Jiang--Su stability is mainly used as a tool to remove the finite extreme boundary condition on $T(\B)$.
The other restrictions can be removed without Jiang--Su stability.
In fact, say that $\phi : \A \rightarrow \Mul(\B)$ is a unital trivial absorbing extension (and we do not assume that $\B$ is $\Z$-stable).
Let $\{ e_n \}_{n=1}^{\infty}$ be an approximate unit for $\B$ with $e_{n+1} e_n = e_n$ for all $n$ and which quasicentralizes $\phi(\A)$ ``fast enough'' (as in \Cref{lem:QuasicentralAU} for a countable dense subset of $\phi(\A)$), and let $\{ \alpha_n \}_{n=1}^{\infty}$ be a sequence in $S^1$ which has a ``close neighbors'' property $\abs{\alpha_n - \alpha_{n+1}} \rightarrow 0$.
Then by \Cref{lem:QuasicentralAU,lem:Apr2720216AM}, if we define $Z =_{df} \sum_{n=1}^{\infty} \alpha_n (e_n - e_{n-1})$, then $\pi(Z)$ is a unitary in $(\pi \circ \phi(\A))' = \paschkedual{\A}{\B}$.
Note that $\{ e_n \}$ need not consist of projections and $Z$ itself need not be a unitary.
Moreover, we can choose the sequence $\{ \alpha_n \}$ so that $\pi(Z)$ is homotopic to $1$ in $(\pi \circ \phi(\A))'$.
Note that here, there are no assumptions that $\A$ be embeddable in $\B$, nor that $\B$ has a nonzero projection.
In fact, we can even start with an arbitrary absorbing extension $\phi$ (and not one in a specific form like the Lin extension).
The problem arises when trying to choose the sequence $\{ \alpha_n \}$ so that $\pi(Z) U$ will be strongly full.
(This is a problem of ``distributing mass over the spectrum of $\pi(Z)U$''.)
With this approach, we would need that $T(\B)$ have finite extreme boundary.

Thus, we need to use a more sophisticated approach with the additional assumption that $\B$ be $\Z$-stable.
Since $\B \cong \B \otimes \Z$, we work with $\B \otimes \Z$.
Firstly, in the new approach, we can again realize the Paschke dual algebra $\paschkedual{\A}{\B}$ using an arbitrary unital absorbing trivial extension $\phi : \A \rightarrow \Mul(\B \otimes \Z)$.
(So $\paschkedual{\A}{\B} = (\pi \circ \phi(\A))'$.)
The operator $Z \in \Mul(\B \otimes \Z)$ that we construct will have the complicated form
\begin{equation}
  \label{equ:ComplicatedZExpr}
  Z = \sum_{k=1}^{\infty} u_k ((e_k - e_{k-1}) \otimes z_k) u_k^*,
\end{equation}
where $\{ e_n \}$ is an approximate unit for $\B$ with $e_{n+1} e_n = e_n$ for all $n$, $z_k \in \Z$ is a unitary and $u_k \in \Mul(\B) \otimes \Z$ is a unitary for all $k$.
These operators will have further properties to be specified below.  

To ensure that $\pi(Z)$ is a unitary, we need to impose conditions on the $u_k$ and $z_k$ in (\ref{equ:ComplicatedZExpr}).
This is essentially \Cref{lem:AUUnitary}, which generalizes the ``close neighbors''
condition of \Cref{lem:Apr2720216AM}.
Note that, again, $\{ e_n \}$ need not consist of projections and $Z$ itself need not be a unitary.

To ensure that $\pi(Z) \in (\pi \circ \phi(\A))' = \paschkedual{\A}{\B}$, we will require that $\{ e_n \}$ quasicentralizes a certain set ``quickly enough'', but also, we are forced to use the $\Z$-stability property.
The complicated main result for this is \Cref{lem:Apr120213AM}, which again imposes conditions on the $u_k$, $z_k$, and $e_k$.
We here roughly summarize some of the key ingredients to \Cref{lem:Apr120213AM}:
\begin{enumerate}
\item\label{item:strict-lower-semicontinuous-extension}
  From standard results about $\Z$-stable \cstar-algebras, we get *-monomorphisms
  $\iota, \Phi : \B \rightarrow \B \otimes \Z$ for which $\iota$ has special form $\iota = id_{\B} \otimes 1_{\Z}$ and $\Phi$ is a *-isomorphism, and we get 
  a sequence of unitaries $\{ u_n \}$ in $\Mul(\B) \otimes \Z$ which witnesses that $\iota$ and $\Phi$ are approximately unitarily equivalent  (\Cref{lem:JiangSuMaps}).
  Moreover, there are unique injective strictly lower semicontinuous extensions $\Mul(\B) \rightarrow \Mul(\B \otimes \Z)$, which we also denote by $\iota$, $\Phi$ respectively (here the extension $\Phi$ will also be an isomorphism).
\item\label{item:quasicentral-approximate-unit}
  Let $S_0 \subseteq \Mul(\B \otimes \Z)$ be a countable set which is dense in $\phi(\A)$, and let $S =_{df} \Phi^{-1}(S_0) \subseteq \Mul(\B)$.
  We require that $\{ e_n \}$ quasicentralizes $S$ at a ``fast rate" as 
  in \Cref{lem:QuasicentralAU}.
\item\label{item:bidiagonalizable}
  We use \Cref{thm:BidiagonalDecomp}, which says that, up to a compact element (i.e., element in $\B \otimes \Z$), every positive operator in $\Mul(\B \otimes \Z)$ is a sum of two diagonal operators (with diagonal entries in $\B \otimes \Z$).
  A key fact is that by \ref{item:strict-lower-semicontinuous-extension} (or \Cref{lem:JiangSuMaps}), for each entry $b$ (in either one of the diagonals), there exists $b' \in \B$ such that $u_n b u_n^* \rightarrow b' \otimes 1_{\Z}$.
\item\label{item:combining} The above \ref{item:strict-lower-semicontinuous-extension}, \ref{item:quasicentral-approximate-unit} and \ref{item:bidiagonalizable} are the main inputs for getting an operator $Z$, as in \eqref{equ:ComplicatedZExpr} so that $\pi(Z) \in (\pi \circ \phi (\A))' = \paschkedual{\A}{\B}$.
As mentioned, the main technique is summarized in \Cref{lem:Apr120213AM}.
\end{enumerate}
It is not hard to see that, with appropriate modification, $\pi(Z) \sim_h 1$ in $(\pi \circ \phi(\A))'$.
More precisely, one moves to $1$ by a path that connects the $z_k$ to $1_{\Z}$.
Thus, there are further restrictions on the $z_k$ which will be chosen to be the unitaries $v_{\Lambda} \in \dimdrop$ which are described two paragraphs before \Cref{lem:Apr720216AM}.

Finally, we need to ensure that $\pi(Z)U$ is strongly full.
As mentioned before, this is the main reason that we require $\Z$-stability (without which, we require the highly restrictive condition that $T(\B)$ has finite extreme boundary), and thus, this is also the main reason for the complicated definition of $Z$ in \eqref{equ:ComplicatedZExpr} and the complicated conditions to establish
$\pi(Z) \in \paschkedual{\A}{\B}$ in \Cref{lem:Apr120213AM} (as discussed in the previous paragraph).
The main results establishing strong fullness are \Cref{lem:MainTechnicalLemma,lem:PreTechnicalLemma}.  Perhaps the most technical lemmas in the paper (\Cref{lem:Apr720216AM,lem:Oct1920186AM,lem:standard-chain-approx-1,lem:standard-chain-approx-2,lem:laurent-uniform-continuity,lem:16u,lem:May2920215PM,lem:Apr1720216AM,lem:17b}) are preliminary results that lead to the central results \Cref{lem:MainTechnicalLemma,lem:PreTechnicalLemma}.
Here is a rough intuition:  To establish strong fullness is a matter of ``distributing mass  over the spectrum of $\pi(Z) U$'' (or, since we have strict comparison, ``distributing tracial mass'').
In the setting of our previous paper \cite{LN-2020-IEOT}, since $Z$ has the form of a diagonal $\sum_j \alpha_j p_j$ and since we have only freedom to vary the constants $\alpha_j$,  we can only ``distribute mass vertically'' and must assume the restrictive condition that $T(\B)$ has finite extreme boundary.
But now that we have $\Z$-stability of $\B$, and now that $Z$ has the form \eqref{equ:ComplicatedZExpr}, we have an ``extra direction'' and can now ``distribute the mass horizontally (along $\Z$) as well as vertically''.
In particular, in the expression \eqref{equ:ComplicatedZExpr}, the unitaries $z_k$ allow us to ``distribute the mass horizontally'' and allow us to achieve the crucial step \Cref{lem:MainTechnicalLemma} with arbitrary metrizable Choquet simplexes $T(\B)$.
Finally, recall that we will be choosing the $z_k$ to have the form $v_{\Lambda} \in \dimdrop$ as in two paragraphs before \Cref{lem:Apr720216AM}.

The above are the main ideas of the argument for $K_1$ injectivity of $\paschkedual{\A}{\B}$.
We finish with some extra remarks, which maybe helpful to the reader.   
Firstly, the technical \Cref{lem:Apr720216AM,lem:Oct1920186AM,lem:standard-chain-approx-1,lem:standard-chain-approx-2,lem:laurent-uniform-continuity,lem:16u,lem:May2920215PM,lem:Apr1720216AM,lem:17b} are essentially at the level of ``hard exercises'' in elementary operator theory;
however, because of the complexity due to many factors, we have opted to write out the details, thus lengthening the paper.
The reader may want to skip these results on the first reading.
Next, because of the generality of our result, there are many additional complicating technical details.
For instance, we note that in \eqref{equ:ComplicatedZExpr}, the sum, for $Z$, is not even a finite sum of diagonal operators.
However, it has sufficiently nice structure for us to be able to work with it;
in particular, the sum is an element of $\mathfrak{S}$, where $\mathfrak{S}$ is the collection of operators defined in the paragraph before \Cref{lem:Apr720216AM};
nonetheless, this leads to an extra layer of complications in our computations.
Finally, we actually prove a result that is slightly more general than what we need to apply \Cref{thm:k1-injectivity-generic} and get $\ko$-injectivity of $\paschkedual{\A}{\B}$.
We actually prove the following:
Say that $U \in \C(\B)$ is a unitary (not necessarily in $(\pi \circ \phi(\A))' = \paschkedual{\A}{\B}$).
Then we can find an operator $Z \in \Mul(\B)$ such that $\pi(Z)$ is a unitary in $\paschkedual{\A}{\B}$, $\pi(Z) \sim_h 1$ in $\paschkedual{\A}{\B}$, and $\pi(Z)U$ is strongly full in $\C(\B)$.

As a further guide, we describe how each lemma contributes to the outline described above.
\begin{itemize}
\item \Cref{lem:AUUnitary} is used to show that the image of \eqref{equ:ComplicatedZExpr} in the corona algebra is unitary.
\item \Cref{lem:QuasicentralAU,lem:Apr120213AM,lem:JiangSuMaps,lem:JiangSuMapsStrictExtension} and \Cref{thm:BidiagonalDecomp} are used to establish that the unitary we construct lies in $\paschkedual{\A}{\B} := (\pi \circ \phi(\A))'$, as well as setting up some of the structure involving the $\Z$-stability of $\B$.
\item \Cref{lem:Apr720216AM,lem:Oct1920186AM} are the foundational results to establish strong fullness.
\item \Cref{lem:standard-chain-approx-1,lem:standard-chain-approx-2,lem:laurent-uniform-continuity,lem:16u,lem:May2920215PM,lem:Apr1720216AM,lem:17b} are technical strict topology arguments involving Laurent polynonmials and approximate units.
\item  \Cref{lem:MainTechnicalLemma} combines \Cref{lem:Apr720216AM,lem:Oct1920186AM} and \Cref{lem:standard-chain-approx-1,lem:standard-chain-approx-2,lem:laurent-uniform-continuity,lem:16u,lem:May2920215PM,lem:Apr1720216AM,lem:17b} in order to provide the key stepping stone to strong fullness.
\item \Cref{lem:PreTechnicalLemma} puts all of the above together to construct the desired unitary $\pi(Z) \in \paschkedual{\A}{\B}$ which is homotopic to $1$ in $\paschkedual{\A}{\B}$ for which $\pi(Z)U$ is strongly full in $\C(\B)$.
\end{itemize}

We note that, until \Cref{lem:MainTechnicalLemma}, we do not use \Cref{lem:AUUnitary,lem:QuasicentralAU,lem:JiangSuMaps,lem:JiangSuMapsStrictExtension,lem:Apr120213AM} and \Cref{thm:BidiagonalDecomp} except at the notation at the beginning of Subsection~\ref{sec:technical-lemmas} (or before \Cref{lem:Apr720216AM}).
In particular, we don't use \Cref{lem:Apr120213AM} until \Cref{lem:PreTechnicalLemma}.

We firstly fix some notation, which will be used for the rest of this paper.
For a separable simple \cstar-algebra $\B$ and for a nonzero element $e \in \Ped(\B)_+$
(the Pedersen ideal of $\B$), we let
$T_e(\B)$ denote the set of all densely defined, norm-lower semicontinuous traces
$\tau$ on $\B_+$ such that $\tau(e) = 1$.  Recall that $T_e(\B)$, with the topology
of pointwise convergence on $\Ped(\B)$, is a compact convex set.  In fact, it is
a Choquet simplex.  All the results and arguments in this paper are independent of
the choice of normalizing element $e \in \Ped(\B)_+ - \{ 0 \}$.  Hence, we will
usually drop the $e$ and simply write $T(\B)$.
Recall also, that every $\tau \in T(\B)$ extends uniquely to a
strictly lower semicontinuous trace
$\Mul(\B)_+ \rightarrow [0, \infty]$, which we will also denote by ``$\tau$''.

Recall next that for all
$\tau \in T(\B)$ and for any $a \in \Mul(\B)_+$,
\begin{equation*}
  d_{\tau}(a) \defeq \lim_{n\rightarrow \infty} \tau(a^{1/n}) \in [0, \infty].
\end{equation*}

For $\delta > 0$,
let $f_{\delta} : [0, \infty) \rightarrow [0,1]$ be the unique continuous function for which
\begin{equation*}
  f_{\delta}(t)
  =
  \begin{cases}
    1 & t \in [\delta, \infty) \\
    0 & t = 0\\
    \makebox{linear on  } & [0, \delta].
  \end{cases}
\end{equation*}

In what follows, for elements $a,b$ in a \cstar-algebra, we use $a \approx_{\epsilon} b$ to denote $\norm{ a-b } < \epsilon$.

\begin{lemma}
  Let $\B$ be a nonunital $\sigma$-unital \cstar-algebra, and let $\{ e_k \}$ be
  an approximate unit for $\B$ for which
  \begin{equation*}
    e_{k+1} e_k = e_k
  \end{equation*}
  for all $k$. Let $\E$ be a unital \cstar-algebra and let $\{ f_k \}$ be an approximate unit for $\B \otimes \E$ for which
  \begin{equation*}
    f_{k+1} f_k = f_k
  \end{equation*}
  for all $k$.

  Let $\{ u_k \}$ be a sequence of unitaries in $\B \otimes \E$,
  $\{ z_k \}$ be a sequence of unitaries in $\E$, and
  let $\{ \epsilon_k \}$ be a sequence in $(0,1)$ for which
  $\sum_{k=1}^{\infty} \epsilon_k < \infty$.

  Suppose that
  \begin{equation}
    \label{eq:e_k-f_k-u_k-epsilon_k-approx}
    u_k (e_j \otimes 1) u_k^* \approx_{\epsilon_k} f_j
  \end{equation}
  for all $j \le k$,
  and suppose that
  for all $\epsilon > 0$,
  there exists an $N \geq 1$ where for all $k \geq N$,
  either
  \begin{equation*}
    u_k = u_{k+1} \makebox{  and  } \norm{ z_k - z_{k+1} } < \epsilon.
  \end{equation*}
  or
  \begin{equation*}
    z_k = z_{k+1} \in \complex 1_{\E}.
  \end{equation*}

  Let
  \begin{equation*}
    u \defeq \sum_{k=1}^{\infty}u_k( (e_k - e_{k-1}) \otimes z_k)u_k^*
  \end{equation*}
  where the sum converges strictly in $\Mul(\B \otimes \E)$ (and where $e_0 \defeq 0$).

  Then $\pi(u)$ is a unitary in $\C(\B \otimes \E)$.
  \label{lem:AUUnitary}
\end{lemma}

\begin{proof}
  Note that the sum defining $u$ converges strictly because of \eqref{eq:e_k-f_k-u_k-epsilon_k-approx}.
  Let $r_k \defeq e_k - e_{k-1}$, and notice that the condition $e_{n+1} e_n = e_n$ implies by induction (and positivity of $e_n$) that $e_n e_k = e_k = e_k e_n$ for all $n > k$.
  Consequently, $r_k r_n = r_n r_k$, and if $\abs{n-k} > 1$, then $r_k r_n = 0$.
  Of course, this also holds for $s_k \defeq f_k - f_{k-1}$.

  By hypothesis,
  \begin{equation*}
    u_k(r_k \otimes 1)u_k^{*} \approx_{2\epsilon_k} s_k.
  \end{equation*}
  Now,
  \begin{equation*}
    u_k (r_k \otimes z_k) u_k^{*} = u_k (r_k \otimes 1) (1 \otimes z_k) u_k^{*} \approx_{2\epsilon_k} s_k u_k (1 \otimes z_k) u_k^{*},
  \end{equation*}
  and similarly,
  \begin{equation*}
    u_k (r_k \otimes z_k) u_k^{*} = u_k (1 \otimes z_k) (r_k \otimes 1) u_k^{*} \approx_{2\epsilon_k} u_k (1 \otimes z_k) u_k^{*} s_k,
  \end{equation*}
  Therefore, there exist contractions $v_k \in \overline{s_k (\B \otimes \E) s_k}$ for which
  \begin{equation*}
    v_k \approx_{12\epsilon_k} u_k (r_k \otimes z_k) u_k^{*}.
  \end{equation*}

  Since $s_k s_n = 0 = s_n s_k$ whenever $\abs{n-k} > 1$, we also know $v_k v_n = 0 = v_n v_k$ whenever $\abs{n-k} > 1$.

  Moreover, for the strictly converging sum $v \defeq \sum_{k=1}^{\infty} v_k$, we know that for all $N$,
  \begin{equation*}
    \pi(v) = \pi\left(\sum_{k=N}^{\infty} v_k\right), \quad\text{and}\quad \pi(u) = \pi\left(\sum_{k=N}^{\infty} u_k (r_k \otimes z_k) u_k^{*} \right).
  \end{equation*}
  Because
  \begin{equation*}
    \norm*{\sum_{k=N}^{\infty} v_k - \sum_{k=N}^{\infty} u_k (r_k \otimes z_k) u_k^{*}} \le \sum_{k=N}^{\infty} \norm{v_k - u_k (r_k \otimes z_k) u_k^{*} } \le \sum_{k=N}^{\infty} 12 in ni \epsilon_k \to 0,
  \end{equation*}
  as $N \to \infty$, we obtain $\pi(v) = \pi(u)$.

  Consequently,
  \begin{align*}
    \pi(u)\pi(u)^{*} = \pi(v v^{*})
    &= \pi\left(\left( \sum_{k=1}^{\infty} v_k \right)\left( \sum_{j=1}^{\infty} v_j^{*} \right)\right) \\
    &= \pi\left( \sum_{k=1}^{\infty} v_k v_k^{*} + v_k v_{k+1}^{*} + v_{k+1} v_k^{*} \right) \\
    &= \pi\left( \sum_{k=1}^{\infty} s_k^2 + v_k v_{k+1}^{*} + v_{k+1} v_k^{*} \right),
  \end{align*}
  where the last inequality follows from an argument similar to the one which shows $\pi(v) = \pi(u)$ along with the estimate
  \begin{equation*}
    v_k v_k^{*} \approx_{24\epsilon_k} u_k (r_k \otimes z_k) u_k^{*} u_k (r_k \otimes z_k^{*}) u_k^{*} = u_k (r_k \otimes z_k) (r_k \otimes z_k^{*}) u_k^{*} = u_k (r_k^2 \otimes 1) u_k^{*} \approx_{4\epsilon_k} s_k^2.
  \end{equation*}

  We also have the estimate
  \begin{equation*}
    v_k v_{k+1}^{*} \approx_{24\epsilon_{k+1}} u_k (r_k \otimes z_k) u_k^{*} u_{k+1} (r_{k+1} \otimes z_{k+1}^{*}) u_{k+1}^{*}.
  \end{equation*}

  Let $\epsilon > 0$, and by hypothesis there is some $N$ such that for all $k \ge N$, either $u_k = u_{k+1}$ and $\norm{z_k - z_{k+1}} < \epsilon$, or $z_k = z_{k+1} \in \complex 1_{\E}$.

  In the former case, $u_k = u_{k+1}$ and $\norm{z_k z_{k+1}^{*} - 1} < \epsilon$ and so,
  \begin{align*}
    u_k (r_k \otimes z_k) u_k^{*} u_{k+1} (r_{k+1} \otimes z_{k+1}^{*}) u_{k+1}^{*}
    &= u_{k+1} (r_k r_{k+1} \otimes z_k z_{k+1}^{*}) u_{k+1}^{*} \\
    &\approx_{\epsilon} u_{k+1} (r_k r_{k+1} \otimes 1) u_{k+1}^{*} \\
    &\approx_{4\epsilon_{k+1}} s_k s_{k+1}.
  \end{align*}

  In the latter case, $z_k = z_{k+1} \in \complex 1_{\E}$ we we find
  \begin{align*}
    u_k (r_k \otimes z_k) u_k^{*} u_{k+1} (r_{k+1} \otimes z_{k+1}^{*}) u_{k+1}^{*}
    &= u_k (r_k \otimes 1) u_k^{*} u_{k+1} (r_{k+1} \otimes 1) u_{k+1}^{*} \\
    &\approx_{4\epsilon_{k+1}} s_k s_{k+1}.
  \end{align*}

  Increasing $N$ if necessary to ensure that $\epsilon_k < \epsilon$ for all $k \ge N$, then we obtain
  \begin{equation*}
    \norm{v_k v_{k+1}^{*} - s_k s_{k+1}} \le 24\epsilon_{k+1} + \epsilon + 4\epsilon_{k+1} < 29 \epsilon.
  \end{equation*}

  This leads to the following estimate on the norm of the sum
  \begin{align*}
    \norm*{ \sum_{k=2N}^{\infty} v_k v_{k+1}^{*} - s_k s_{k+1} }
    &\le \norm*{ \sum_{k=N}^{\infty} v_{2k} v_{2k+1}^{*} - s_{2k} s_{2k+1} } \\
    &\qquad + \norm*{ \sum_{k=N}^{\infty} v_{2k+1} v_{2k+2}^{*} - s_{2k+1} s_{2k+2} } \\
    &\le \sup_{k \ge N} \norm{v_{2k} v_{2k+1}^{*} - s_{2k} s_{2k+1}} \\
    &\qquad + \sup_{k \ge N} \norm{v_{2k+1} v_{2k+2}^{*} - s_{2k+1} s_{2k+2}} \\ &\le 58\epsilon,
  \end{align*}
  where the second inequality follows because $v_{2k} v_{2k+1}^{*} - s_{2k} s_{2k+1}$ live in orthogonal subalgebras as $k$ varies, and similarly for $v_{2k+1} v_{2k+2}^{*} - s_{2k+1} s_{2k+2}$.

  Since $\{f_k\}$ is an approximate unit for $\B \otimes \E$, $\sum_{k=1}^{\infty} s_k = 1$ with convergence in the strict topology.
  Therefore,
  \begin{equation*}
    1 = 1^2 = \sum_{k=1}^{\infty} s_k^2 + 2 s_k s_{k+1}.
  \end{equation*}

  Finally, we note that
  \begin{align*}
    \pi(uu^{*} - 1)
    &= \pi\left( \sum_{k=1}^{\infty} (s_k^2 + v_k v_{k+1}^{*} + v_{k+1} v_k^{*}) - (s_k^2 + 2 s_k s_{k+1}) \right) \\
    &= \pi \left( \sum_{k=2N}^{\infty} v_k v_{k+1}^{*} + v_{k+1} v_k^{*} - 2 s_k s_{k+1} \right).
  \end{align*}
  Therefore,
  \begin{equation*}
    \norm{\pi(uu^{*} - 1)} \le \norm*{ \sum_{k=N}^{\infty} v_k v_{k+1}^{*} + v_{k+1} v_k^{*} - 2 s_k s_{k+1} } \le 116\epsilon.
  \end{equation*}
  Since $\epsilon$ is arbitrary, $\norm{\pi(uu^{*} - 1)} = 0$.
  The same argument applied to $u^{*}$ proves $\norm{\pi(u^{*}u - 1)} = 0$, and hence $\pi(u)$ is unitary in $\C(\B \otimes \E)$.
\end{proof}

The next computation should be well-known, but we provide it for the convenience of the
reader.

\begin{lemma}
  \label{lem:QuasicentralAU}
  Let $\B$ be a nonunital $\sigma$-unital simple \cstar-algebra, and let
  $S \defeq \{ x_k \mid k \geq 1 \} \subseteq \Mul(\B)$
  be a countable set.

  Then there exists an approximate unit $\{ e_n \}$ for $\B$ such that
  \begin{equation*}
    e_{n+1} e_n = e_n
  \end{equation*}
  for all $n$,
  and
  $\{ e_n \}$ quasicentralizes $S$ in the following strong sense:
  For all $n$, for all $1 \leq k \leq n$,
  \begin{equation*}
    \norm{ e_n x_k - x_k e_n } < \frac{1}{2^n}.
  \end{equation*}

  Moreover, for any $\{ e_n \}$ as above, we have that
  \begin{equation*}
    \pi\left(\sum_{n=1}^{\infty} \alpha_n (e_n - e_{n-1})\right) \in \pi(S)'
  \end{equation*}
  for every bounded sequence $\{ \alpha_n \}$ of complex numbers.
\end{lemma}

\begin{proof}

  Let $S \defeq \{ x_k : k \geq 1 \}$.

  Let $\{ e_n \}$ be an approximate unit for $\B$ such that
  \begin{equation*}
    e_{n+1} e_n = e_n
  \end{equation*}
  for all $n$, and $\{ e_n \}$ quasicentralizes $\{ x_k \}$ (see \cite[Theorem~2.2]{Ped-1990});
  that is, for all $k$,
  \begin{equation*}
    \norm{ e_n x_k - x_k e_n } \rightarrow 0
  \end{equation*}
  as $n \rightarrow \infty$.

  Passing to a subsequence of $\{ e_n \}$ if necessary, we may assume that for all $n$, for all $1 \leq k \leq n$,
  \begin{equation*}
    \norm{ e_n x_k - x_k e_n } < \frac{1}{2^n}.
  \end{equation*}

  For all $n$, let
  \begin{equation*}
    r_n \defeq e_n - e_{n-1}.
  \end{equation*}

  Then, for each $k \in \nat$ let $b_k \defeq \sum_{n=k+1}^{\infty} \alpha_n r_n$, and notice
  \begin{equation*}
    x_k b_k - b_k x_k = \sum_{n=k+1}^{\infty} \alpha_n \big( (x_k e_n - e_n x_k) + (e_{n-1} x_k - x_k e_{n-1} ) \big).
  \end{equation*}
  But for $n > k$,
  \begin{equation*}
    \norm{(x_k e_n - e_n x_k) + (e_{n-1} x_k - x_k e_{n-1} )} \le \frac{3}{2^n},
  \end{equation*}
  and the sequence $\alpha_n$ is bounded.
  Hence the sum representing $x_k b_k - b_k x_k$ converges in norm, so $x_k b_k - b_k x_k \in \B$.
  Consequently, $x_k \left(\sum_{n=1}^{\infty} \alpha_n r_n \right) - \left(\sum_{n=1}^{\infty} \alpha_n r_n \right) x_k \in \B$,
  and therefore $\pi\left(\sum_{n=1}^{\infty} \alpha_n (e_n - e_{n-1})\right)$ commutes with each $\pi(x_k)$.
\end{proof}

We will need to use the following bidiagonal decomposition theorem:

\begin{theorem}
  Let $\B$ be a $\sigma$-unital \cstar-algebra, and let $\{ e_n \}_{n=1}^{\infty}$ be an approximate
  unit for $\B$ for which
  \begin{equation*}
    e_{n+1} e_n = e_n
  \end{equation*}
  for all $n \geq 1$.

  Let $\F \subset \Mul(\B)$ be a finite collection of positive elements and let $\epsilon > 0$ be given.
  Then there exist two subsequences $\{ m_k \}$ and $\{ n_k \}$ of the nonnegative integers with
  \begin{equation*}
    n_{k+1} < m_{k} < n_{k+2} < m_{k+1}
  \end{equation*}
  for all $k$ such that for all
  $X \in \F$, there exists a bounded sequence $\{ d_k \}$ of positive elements of $\B$, and
  there exists some $b \in \B_{SA}$ such that the following statements are true:

  \begin{enumerate}
  \item  $X = \sum_{k=1}^{\infty} d_k + b$, where the sum converges strictly.
  \item  $\norm{ b } < \epsilon$.
  \item  $d_k \in \overline{(e_{m_k} - e_{n_k}) \B (e_{m_k} - e_{n_k})}$, for all $k$.
  \end{enumerate}
  In the above, we define $n_1 \defeq 0$ and $e_0 \defeq 0$.

  Note that the above implies that for all $|k - l| \geq 2$,
  $d_k \perp d_l$. I.e., the above series is a bidiagonal series.
  \label{thm:BidiagonalDecomp}
\end{theorem}

\begin{proof}
  This is from \cite{KNZ-2017-CJM} Theorem 4.2, Remark 4.3 and their proofs.
\end{proof}

The Jiang--Su algebra $\Z$ \cite{JS-1999-AJM} is the unique simple unital nonelementary inductive limit of dimension drop algebras
with $\k$-theory invariant being the same as the complex numbers $\complex$, i.e.,
\begin{equation*}
  (\kz(\Z), \kz(\Z)_+, \ko(\Z), T(\Z)) = (\ints, \ints_+, 0, \{ pt \})
  = (\kz(\complex), \kz(\complex)_+, \ko(\complex), T(\complex)).
\end{equation*}
We let $\tau_{\Z}$ denote the unique tracial state of $\Z$.

Like the complex numbers $\complex$, $\Z \otimes \Z \cong \Z$.
A \cstar-algebra $\D$ is said to be \term{$\Z$-stable}  (or \term{Jiang--Su-stable}) if
$\D \cong \D \otimes \Z$.
Jiang--Su-stability is a regularity property which is
an axiom in the classification program for simple
amenable \cstar-algebras.
Indeed, if $\B$ is a simple $\Z$-stable \cstar-algebra, then either $\B$ is stably finite or purely infinite.
Also, if $\B$ is an exact simple $\Z$-stable \cstar-algebra, then $\B$ has \term{strict comparison of positive elements};
that is, for every $a,b \in (\B \otimes \K)_+$,
\begin{equation*}
  \text{if } d_{\tau}(a) < d_{\tau}(b) \text{ or } d_{\tau}(b) = \infty, \  \forall \tau \in T(\B), \text{ then } a \preceq b.
\end{equation*}
In fact, for simple separable nuclear \cstar-algebras, it is a question whether $\Z$-stability is equivalent to strict comparison of positive elements.

Good references for the Jiang--Su algebra and some of its basic properties are \cite{JS-1999-AJM,GJS-2000-CMB,Roer-2004-IJM}.

\begin{lemma}
  \label{lem:JiangSuMaps}
  Let $\B$ be a $\Z$-stable \cstar-algebra and let $\iota : \B \rightarrow \B \otimes \Z$ be the *-embedding
  given by
  \begin{equation*}
    \iota : d \mapsto d \otimes 1_{\Z}.
  \end{equation*}

  Then there exists a *-isomorphism $\Phi : \B \rightarrow \B \otimes \Z$ such that
  $\iota$ and $\Phi$ are approximately unitarily equivalent, i.e., there exists a sequence
  $\{ u_n \}$ of unitaries in $\Mul(\B \otimes \Z)$ for which
  \begin{equation*}
    u_n \iota(a) u_n^* \rightarrow \Phi(a)
  \end{equation*}
  for all $a \in \B$.
\end{lemma}

\begin{proof}
  This follows from \cite{JS-1999-AJM} Theorems 7.6 and 8.7.
\end{proof}

\begin{lemma}
  Let $\B$ be a $\sigma$-unital
  $\Z$-stable \cstar-algebra and let $\iota, \Phi, \makebox{  and  } \{ u_n \}$ be as
  in \Cref{lem:JiangSuMaps}.

  $\iota$ (respectively $\Phi$)
  extends uniquely to a *-embedding (respectively *-isomorphism)
  $\Mul(\B) \hookrightarrow \Mul(\B
  \otimes \Z)$ which is strictly continuous on bounded sets, which we also denote by $\iota$ (respectively $\Phi$).

  Then
  \begin{equation*}
    u_n \iota(X) u_n^* \rightarrow \Phi(X)
  \end{equation*}
  and
  \begin{equation*}
    u_n^* \Phi(X) u_n \rightarrow \iota(X)
  \end{equation*}
  in the strict topology on $\Mul(\B \otimes \Z)$, for all $X \in \Mul(\B)$.
  \label{lem:JiangSuMapsStrictExtension}
\end{lemma}

\begin{proof}
  Let $\{ e_n \}$ be an approximate unit for
  $\B$ for which
  \begin{equation*}
    e_{n+1} e_n = e_n
  \end{equation*}
  for all $n$.

  For all $n$, let $f_n \defeq \Phi(e_n) \in \B \otimes \Z$. Then $\{ f_n \}$
  is an approximate unit for $\B \otimes \Z$ and
  \begin{equation*}
    f_{n+1} f_n = f_n
  \end{equation*}
  for all $n$.

  Let $X \in \Mul(\B)$ be arbitrary.
  We may assume that $X \geq 0$ and $\norm{ X } = 1$.
  Let $\epsilon > 0$ be given, and let $N \geq 1$ be given.

  By \Cref{thm:BidiagonalDecomp}, there exist to subsequences
  $\{ m_k \}$ and $\{ n_k \}$ of the nonnegative integers with
  \begin{equation*}
    n_{k+1} < m_k < n_{k+2}  < m_{k+1}
  \end{equation*}
  for all $k$, such that there exists a sequence
  $\{ d_k \}$ of positive contractive elements of $\B$
  and $b \in \B$ for which the following statements are true:
  \begin{enumerate}
  \item[i.] $X = \sum_{k=1}^{\infty} d_k + b$ where the sum
    converges strictly.
  \item[ii.] $\norm{ b } < \frac{\epsilon}{10}$.
  \item[iii.] $d_k \in \her_{\B}(e_{m_k} - e_{n_k})$, for all $k$.
  \end{enumerate}

  Choose $K \geq 1$ so that $n_K > N + 1$.

  Choose $L \geq 1$ so that for all $l \geq L$
  \begin{equation*}
    u_l \left( \sum_{k=1}^{K+1} d_k \otimes 1 \right) u_l^*
    \approx_{\frac{\epsilon}{10}} \sum_{k=1}^{K+1} \Phi(d_k)
  \end{equation*}

  and

  \begin{equation*}
    u_l (e_{m_K} \otimes 1) u_l^* \approx_{\frac{\epsilon}{10}}
    f_{m_K}.
  \end{equation*}

  So for all $l \geq L$,

  \begin{align*}
    u_l \left( \sum_{k=1}^{\infty} d_k \otimes 1 \right) u_l^* f_N
    &= u_l \left( \sum_{k=1}^{\infty} d_k \otimes 1 \right) u_l^*
          f_{m_K} f_N \\
    &\approx_{\frac{\epsilon}{10}} u_l \left( \sum_{k=1}^{\infty} d_k \otimes 1 \right)(e_{m_K} \otimes 1)  u_l^* f_N \\
    &= u_l \left( \sum_{k=1}^{K+1} d_k \otimes 1 \right)(e_{m_K} \otimes 1) u_l^* f_N \\
    &\approx_{\frac{\epsilon}{5}}  \sum_{k=1}^{K+1} \Phi(d_k) f_N \\
    &= \sum_{k=1}^{\infty} \Phi(d_k) f_N \\
    &= \Phi\left(\sum_{k=1}^{\infty} d_k \right) f_N.
  \end{align*}

  Since $\norm{ b } < \frac{\epsilon}{10}$, for all $l \geq L$,
  \begin{equation*}
    u_l \iota(X) u_l^* f_N \approx_{\epsilon} \Phi(X) f_N.
  \end{equation*}

  Taking the adjoint of the above, we have that for all $l \geq L$,
  \begin{equation*}
    f_N u_l \iota(X) u_l^* \approx_{\epsilon} f_N \Phi(X).
  \end{equation*}

  By a similar argument, we can find $L' \geq 1$ such that for all
  $l \geq L'$,

  \begin{equation*}
    \iota(X) (e_N \otimes 1) \approx_{\epsilon} u_l^* \Phi(X) u_l (e_N
    \otimes 1)
  \end{equation*}
  and

  \begin{equation*}
    (e_N \otimes 1) \iota(X) \approx_{\epsilon}
    (e_N \otimes 1) u_l^* \Phi(X) u_l.
  \end{equation*}

  Since $\epsilon, N, X$ were arbitrary, we are done.
\end{proof}

The next lemma is our key result for constructing a unitary in the Paschke dual.
We refer to it in order to establish notation in Subsection~\ref{sec:technical-lemmas} but we do not use it again until \Cref{lem:PreTechnicalLemma}.

\begin{lemma}
  Let $\B$ be a nonunital $\sigma$-unital simple $\Z$-stable \cstar-algebra, and let
  $\iota, \Phi, \makebox{  and  } \{ u_k \}$ be as in \Cref{lem:JiangSuMaps}.
  Let
  $S_0 \subseteq \Mul(\B \otimes \Z)$ be a countable subset,
  and let $\{ e_n \}$ be an approximate unit for $\B$ which
  quasicentralizes $S \defeq \Phi^{-1}(S_0)$
  as in \Cref{lem:QuasicentralAU}.

  Let $\F \subseteq S_0$ be a finite subset.
  Then for every $\epsilon > 0$, there exists an $N \geq 1$ such that
  for all $n \geq N$ and for all $L \geq 1$, there exists $K \geq 1$ where
  for all contractive elements $z_0, ...., z_{L} \in \Z$,
  and for all $k \geq K$,

  \begin{equation*}
    \left[u_k \left(\sum_{j=0}^{L} (e_{n +j} - e_{n + j -1}) \otimes z_j\right) u_k^*,X
    \right] \approx_{\epsilon} 0
  \end{equation*}
  for all $X \in \F$.

  In the above, we are using the standard additive commutator
  $[x,y] \defeq xy - yx$.

  \label{lem:Apr120213AM}
\end{lemma}

\begin{proof}
  Let $\Phi^{-1}(\F) \defeq \{ X_1, ..., X_M \}$.
  We may assume that each $X_j$ is positive and has norm one.

  Choose $N \geq 1$ so that
  \begin{equation*}
    \sum_{j=N - 1}^{\infty} \frac{1}{2^j} < \frac{\epsilon}{20}.
  \end{equation*}

  By \Cref{lem:QuasicentralAU}, for all $n' \geq N-1$,
  for all $1 \leq j \leq M$,
  \begin{equation*}
    \norm{ e_{n'} X_j - X_j e_{n'} } < \frac{1}{2^{n'}}.
  \end{equation*}

  Now let $n \geq N$ and $L \geq 1$ be given.

  Hence,  for all $1 \leq l \leq M$,

  \begin{equation}
    \left[  \sum_{j=0}^L (e_{n+j} - e_{n+j-1}) \otimes z_j, X_l \otimes 1 \right]
    \approx_{\frac{\epsilon}{20}} 0.
    \label{equ:Apr120212AM}
  \end{equation}

  For all $n'$, let
  \begin{equation*}
    f_{n'} \defeq \Phi(e_{n'}).
  \end{equation*}

  By \Cref{lem:JiangSuMapsStrictExtension},
  choose $K \geq 1$ such that
  for all $k \geq K$,
  for all $1 \leq l \leq M$, and for all $0 \leq n' \leq n_L + 2$,
  \begin{equation*}
    u_k (X_l \otimes 1) u_k^* f_{n'} \approx_{\frac{\epsilon}{10(L+1)}}
    \Phi(X_l) f_{n'}
  \end{equation*}
  and
  \begin{equation*}
    u_k (e_{n'} \otimes 1) u_k^* \approx_{\frac{\epsilon}{10(L+1)}}
    f_{n'}.
  \end{equation*}

  Then for all $k \geq K$, conjugating \eqref{equ:Apr120212AM} by
  $u_k$, we have that
  \begin{equation*}
    \left[  \sum_{j=0}^L u_k ((e_{n+j} - e_{n+j-1}) \otimes z_j)u_k^*,
      \Phi(X_l) \right]
    \approx_{\epsilon} 0.
  \end{equation*}
\end{proof}

\subsection{Technical lemmas}
\label{sec:technical-lemmas}

For this subsection, we begin with a series of technical lemmas which are used exclusively to prove \Cref{lem:PreTechnicalLemma}.
In a first reading of this paper, after reading the standing assumptions and notations which immediately follow, the reader may wish to skip directly to \Cref{lem:PreTechnicalLemma}
(the \Cref{lem:standard-chain-approx-1,lem:standard-chain-approx-2,lem:laurent-uniform-continuity,lem:16u,lem:May2920215PM,lem:Apr1720216AM,lem:17b,lem:MainTechnicalLemma} are especially technical).\\

\textbf{We next fix some notation which will be used for the rest of this subsection (3.1).  We present this notation here, since it is only for this subsection.} \\

Let $\B$ be a nonunital separable simple stable and
$\Z$-stable \cstar-algebra.
Let
\begin{equation*}
  \iota, \Phi : \B \rightarrow \B \otimes \Z
\end{equation*} be the *-monomorphisms
and $\{ u_k \}$ a sequence of unitaries in $\Mul(\B \otimes \Z)$ as in \Cref{lem:JiangSuMaps}, and, as in \Cref{lem:JiangSuMapsStrictExtension}, we use the same notation to denote the strictly continuous (on bounded sets) extensions
\begin{equation*}
  \iota, \Phi : \Mul(\B) \rightarrow \Mul(\B \otimes \Z).
\end{equation*}

We fix a strictly decreasing
sequence  $\{ \epsilon_j \}$  in $(0,1)$ for which
\begin{equation*}
  \sum_{j=1}^{\infty} \epsilon_j < 1.
\end{equation*}

Let $S_0 \subset \Mul(\B \otimes \Z)$ be a countable set of contractions.
Let
$\{ e_k \}$ be an approximate unit for $\B$ which quasicentralizes
$S \defeq \Phi^{-1}(S_0)$ as in \Cref{lem:QuasicentralAU}.
Say that $S_0$ can be enumerated as $S_0 \defeq \{ T'_j \}_{j=1}^{\infty}$.
For each $k \geq 1$, plug $\{ T'_1, ..., T'_k \}$ and $\epsilon_k$ into
\Cref{lem:Apr120213AM} to get an integer $N_k' \geq 1$.
We may assume that $\{ N_k' \}$ is strictly increasing and hence, is a
subsequence of the positive integers.

For all $k$,
let
\begin{equation*}
  f_k \defeq \Phi(e_k) \in \B \otimes \Z.
\end{equation*}
Note that $\{f_k\}$ is an approximate unit for $\B \otimes \Z$ which quasicentralizes $S_0$ and $f_k f_{k+1} = f_k$ for all $k \in \nat$.

For all $k$,
let
\begin{equation*}
  r_k \defeq e_k - e_{k-1}
\end{equation*}
and
\begin{equation*}
  s_k \defeq f_k - f_{k-1},
\end{equation*}
and for all
$m \leq n$, let
\begin{equation*}
  r_{m,n} \defeq \sum_{k=m}^n r_k = \sum_{k=m}^n (e_k - e_{k-1}) =
  e_n - e_{m-1}
\end{equation*}
and
\begin{equation*}
  s_{m,n} \defeq \sum_{k=m}^n s_k = f_n - f_{m-1}.
\end{equation*}

Note that neither the $e_k$, $f_k$, $r_k$  nor the $s_k$ need to be projections.

Let $U \in \C(\B \otimes \Z)$ be a unitary and
let $T \in \Mul(\B \otimes \Z)$ be a contraction such that
\begin{equation*}
  \pi(T) = U.
\end{equation*}

Also, we let $\overline{B(0,1)}$ denote the closed unit ball of the
complex plane, i.e., $\overline{B(0,1)} \defeq \{ \alpha \in
\complex : |\alpha| \leq 1 \}$.

Recall also that a \term{Laurent polynomial} on the punctured closed
disk  $\overline{B(0,1)} - \{ 0 \}$ is a
continuous function $h : \overline{B(0,1)} - \{ 0 \} \rightarrow \complex$ which has the form
\begin{equation*}
  h(\lambda) = \sum_{n= 0}^N \beta_n \lambda^n  + \sum_{m=1}^M \gamma_m \overline{\lambda}^m
\end{equation*}
for all $\lambda \in \overline{B(0,1)} - \{ 0 \}$.
Here, $M, N \geq 1$ are integers and $\beta_n,\gamma_m \in \complex$.

In what follows, we will use that the algebra of Laurent polynomials, when
restricted to the circle $S^1$, is uniformly dense in
$C(S^1)$.

Next, if $\C$ is a unital \cstar-algebra and $h$ is a Laurent polynomial as
above, then for all contractive $x \in \C$, we define
\begin{equation*}
  h(x) \defeq \sum_{n=0}^N \beta_n x^n + \sum_{m=1}^M \gamma_m (x^*)^m.
\end{equation*}
This is well-defined by the uniqueness of Laurent series expansion.  Note that
when $x$ is a unitary, this is consistent with the continuous functional
calculus.

For a real-valued function $f$, we let $\osupp(f) \defeq f^{-1}(\real\setminus\{0\})$ denote the \term{open support} of $f$.
Of course, $\overline{\osupp(f)} = \supp(f)$.

We have a unital \cstar-subalgebra $\dimdrop \subset \Z$ which is a dimension-drop
algebra, such that the spectrum $\widehat{\dimdrop} = [0,1]$.
Moreover, we can assume that the Borel probability measure $\mu_{\tau_{\Z}}$ on $[0,1]$, induced by $\tau_{\Z}$ is Lebesgue measure.
See \cite{JS-1999-AJM} for reference to the properties listed above and others.

For each $n \geq 1$, let $I_{n,1}, ... , I_{n,n},
I'_{n,1}, ..., I'_{n,n}, I''_{n,1}, ..., I''_{n,n}$ be
open subintervals of $[0,1]$ such that

\begin{equation*}
  \sup I_{n,j} < \inf I_{n, j+1} \makebox{  for all  } n, j,
\end{equation*}

\begin{equation*}
  0, 1 \notin \overline{I_{n,j}} \makebox{  and  } \overline{I_{n, j}} \cap \overline{I_{n,k}} = \emptyset \makebox{  for all  }j \neq k,
\end{equation*}

\begin{equation*}
  \overline{I''_{n,j}} \subset I'_{n,j} \subset \overline{I'_{n,j}}
  \subset I_{n,j}
  \makebox{  for all } j,
\end{equation*}

and

\begin{equation*}
  \mu_{\tau_{\Z}}(I''_{n,j})  >  \frac{1}{2n}
  \makebox{  for all  } j.
\end{equation*}

For each $n \geq 1$ and $1 \leq j \leq n$, let
$z_{n,j}, z'_{n,j} : [0,1] \rightarrow [0,1]$ be continuous functions for
which
\begin{equation}
  \osupp(z'_{n,j}) = I'_{n,j}
  \makebox{  and  }
  z'_{n,j}(t) = 1 \makebox{  for all  } t \in I''_{n,j},
  \label{equ:Apr1020216AM}
\end{equation}
and also,
\begin{equation}
  \osupp(z_{n,j}) = I_{n,j} \makebox{  and  }
  z_{n,j}(t) = 1 \makebox{ for all } t \in I'_{n,j} \makebox{  and  }
  j.
  \label{Apr1020216:01AM}
\end{equation}

Clearly,
\begin{equation*}
  z_{n,j}, z'_{n,j} \in \dimdrop \subset \Z.
\end{equation*}

Note that for all $n, j$,
\begin{equation*}
  \tau_{\Z}(z'_{n,j}) \geq \frac{1}{2n} \makebox{  and  }
  z_{n,j} z'_{n,j} = z'_{n,j}.
\end{equation*}

Let $\lambda_1, \ldots, \lambda_m \in [0,1)$ and suppose that
\begin{equation*}
  0 \leq \lambda_1 < \lambda_2 < \cdots < \lambda_m.
\end{equation*}
Let
\begin{equation*}
  \Lambda \defeq (\lambda_1, \ldots, \lambda_m).
\end{equation*}
Define a positive element $h_{\Lambda} \in \dimdrop \subset \Z$ such that
\begin{equation*}
  h_{\Lambda}(0) = 0, \text{ and } h_{\Lambda}(1) < 1.
\end{equation*}
and $h_{\Lambda}$ is increasing
\begin{equation*}
  h_{\Lambda} (t) = \lambda_j \makebox{ for all } t \in I_{m,j}, \makebox{ and }
  1 \leq j \leq m.
\end{equation*}
Then the unitary $v_{\Lambda} \in \dimdrop \subset \Z$ defined by
\begin{equation*}
  v_{\Lambda} \defeq e^{2\pi i h_{\Lambda}}
\end{equation*}
is path-connected to $1$ via a norm-continuous path of unitaries $v_{\Lambda,s} := e^{2\pi i (1-s) h_{\Lambda}} \in \dimdrop$ ($0 \le s \le 1$) with length at most $2\pi$.

Let
$\mathfrak{S} \subset \Mul(\B \otimes \Z)$ consist of all strict converging sums of the form
$\sum_{j=1}^{\infty} a_j$, where for all $j$,
$a_j \in \B \otimes \Z$ is a contractive element for which
\begin{equation*}
  \dist(a_j, \overline{(f_j - f_{j-1}) (\B \otimes \Z) (f_j - f_{j-1})}) < \epsilon_j.
\end{equation*}

Note that for every $\sum_j a_j \in \mathfrak{S}$,
\begin{equation*}
  \norm3{ \sum_j a_j } < 3.
\end{equation*}

\begin{lemma}
  Suppose that $\B$ is stably finite.
  Let $\epsilon > 0$ be given.
  Let $A_1, ..., A_n, A'_1, ... A'_n \in \Mul(\B)_+$ be contractive elements
  such that
  $\Ideal\{A'_1, ..., A'_n\} = \Mul(\B)$ and
  \begin{equation*}
    A_j (A'_j)^{1/l} \approx_{\epsilon} (A'_j)^{1/l}
  \end{equation*}
  for all $l \geq 1$.

  Then for every contractive positive element $a_0 \in \B \otimes \Z$,
  there exists an element $x \in \B \otimes \Z$ with $\norm{ x } \leq 2$.

  such that
  \begin{equation*}
    x \left( \sum_{j=1}^n A_j \otimes z_{n,j} \right) x^* \approx_{2 \epsilon}
    a_0.
  \end{equation*}
  \label{lem:Apr720216AM}
\end{lemma}

\begin{proof}
  We may assume that $\epsilon < \frac{1}{10}$.

  We have that
  \begin{equation*}
    \tau\left( \sum_{j=1}^n A'_j \otimes 1_{\Z} \right) = \infty
  \end{equation*}
  for all $\tau \in T(\B \otimes \Z).$

  Hence,
  \begin{equation*}
    \tau\left( \sum_{j=1}^n A'_j \otimes z'_{n,j} \right) = \infty
  \end{equation*}
  for all $\tau \in T(\B \otimes \Z).$

  Let us simplify notation by letting
  \begin{equation*}
    A' \defeq \sum_{j=1}^n A'_j \otimes z'_{n,j}
  \end{equation*}
  and
  \begin{equation*}
    A \defeq \sum_{j=1}^n A_j \otimes z_{n,j}.
  \end{equation*}

  Note that $A, A' \in \Mul(\B)_+$ are contractive elements for which
  \begin{equation*}
    A (A')^{1/l} \approx_{\epsilon} (A')^{1/l}
  \end{equation*}
  for all $l \geq 1$.

  Hence, if $a \in \overline{A' (\B \otimes \Z) A'}$
  is any strictly positive element,
  \begin{equation*}
    d_{\tau}(a) = \infty.
  \end{equation*}

  Hence, since $\B \otimes \Z$ has strict comparison,
  \begin{equation*}
    a_0 \preceq a.
  \end{equation*}

  Hence, there exists $y \in \B \otimes \Z$ for which
  \begin{equation*}
    a_0 \approx_{\frac{\epsilon}{10}} y  A' y^*.
  \end{equation*}

  Note that
  \begin{equation*}
    \norm{ y (A')^{1/2} }^2 \approx_{\frac{\epsilon}{10}} \norm{ a_0 } \leq 1.
  \end{equation*}

  Choose $L \geq 1$ so that
  for all $l \geq L$,
  \begin{equation*}
    y (A')^{1/2} (A')^{1/l} \approx_{\frac{\epsilon}{10}} y (A')^{1/2}.
  \end{equation*}

  Taking
  \begin{equation*}
    x \defeq y (A')^{1/2} (A')^{\frac{1}{2L}},
  \end{equation*}
  we have that
  \begin{equation*}
    \norm{ x } \leq 2
  \end{equation*}
  and
  \begin{equation*}
    x A x^* \approx_{5 \epsilon}  a_0
  \end{equation*}
  as required.
\end{proof}

\begin{lemma}
  \label{lem:Oct1920186AM}
  Suppose that $A, A', A'' \in \C(\B)_+$ are contractive elements and $\delta > 0$ such that
  \begin{equation*}
    A A' = A'
  \end{equation*}
  and
  \begin{equation*}
    A'' \in \her((A' - \delta)_+).
  \end{equation*}

  Let $A_0 \in \Mul(\B)$ be any contractive lift of $A$, and let $\epsilon > 0$ be given.

  Then for every $M \geq 0$, there exists an $A''_0 \in \overline{(1 - e_M)\Mul(\B)(1 -  e_M)}$ which is a contractive positive lift of $A''$ such that for all $l \geq 1$,
  \begin{equation*}
    A_0 (A''_0)^{1/l} \approx_{\epsilon} (A''_0)^{1/l} \approx_{\epsilon} (A''_0)^{1/l} A_0.
  \end{equation*}
\end{lemma}

\begin{proof}
  The proof is exactly the same as that of
  \cite{LN-2020-IEOT} Lemma 4.1, except that $e_M$ is no longer a projection
  and $e_M^{\perp}$ is replaced with $1 - e_M$.
\end{proof}

All the \Cref{lem:standard-chain-approx-1,lem:standard-chain-approx-2,lem:laurent-uniform-continuity,lem:16u,lem:May2920215PM,lem:Apr1720216AM,lem:17b,lem:MainTechnicalLemma} are technical and messy lemmas used only in the proof of \autoref{lem:PreTechnicalLemma}.
We introduce some new terminology for these lemmas:
A series $\sum_j a_j \in \mathfrak{S}$ is \term{standard} if $a_j \in \her_{\B \otimes \Z}(s_j)$ for all $j$.

\begin{lemma}
  \label{lem:standard-chain-approx-1}
  \hfill
  \begin{enumerate}
  \item \label{item:fmXfL-approx-fmX} For every $\epsilon > 0$, for every $X  \in \Mul(\B \otimes \Z)$, for every $m \geq 1$, there exists an $L \geq 1$ such that
    \begin{equation*}
      f_m X  f_L \approx_{\epsilon} f_m X.
    \end{equation*}
  \item \label{item:fmXfL-equal-fmX}  For every $m \geq 1$, there exists an $L \geq 1$ such for every  standard $X \in \mathfrak{S}$,
    \begin{equation*}
      f_m X  f_L = f_m X.
    \end{equation*}
    (In particular, $L \geq m+2$ will do.)
  \item For every $m,n \geq 1$, for every $\epsilon > 0$, for any finite collection $Y_1,\ldots,Y_l \in \Mul(\B \otimes \Z)$, there exist $m_1,\ldots,m_n,L$ such that
    for all $X_1, \ldots, X_n \in \Mul(\B \otimes \Z)$ where for all $1 \leq j \leq n$,
    either $X_j = Y_i$ for some $1 \le i \le l$, or $X_j \in \mathfrak{S}$ is standard,
    \begin{equation*}
      f_m X_1 X_2 \cdots X_n \approx_{\epsilon} f_m X_1 f_{m_1} X_2 f_{m_2} \cdots X_n f_{m_n},
    \end{equation*}
    and
    \begin{equation*}
      f_m X_1 X_2 \cdots X_n (1 - f_L) \approx_{\epsilon} 0.
    \end{equation*}
  \end{enumerate}
\end{lemma}

\begin{proof}
  \hfill
  \begin{enumerate}
  \item Immediate since $f_m X \in \B \otimes \Z$ and $\set{f_k}$ is an approximate unit.
  \item Immediate because $f_{k+1} f_k = f_k$ for all $k$, and since $X$ is standard, $f_m X \in \her_{\B \otimes \Z}(f_{m+1})$.
  \item Proceed by induction on $n$ to establish that for every $\epsilon > 0$, there is a sequence $m_1, \ldots, m_n$ such that
    \begin{equation*}
      f_m X_1 X_2 \cdots X_n \approx_{\epsilon}f_m X_1 f_{m_1} X_2 f_{m_2} \cdots X_n f_{m_n}.
    \end{equation*}
    To finish the proof, set $L \defeq m_n + 2$.

    For the base case, let $\epsilon > 0$, and obtain $L_1,\ldots,L_l$ by applying \ref{item:fmXfL-approx-fmX} to each $Y_i$, $1 \le i \le l$.
    Then setting $m_1 \defeq \max \set{m+2,L_1,\ldots,L_l}$ is sufficient.

    For the inductive step, let $\epsilon > 0$, and set $B \defeq \max \set{3,\norm{Y_1},\ldots,\norm{Y_l}}$, and suppose that we have $m_1,\ldots,m_k$ such that
    \begin{equation*}
      f_m X_1 X_2 \cdots X_k \approx_{\frac{\epsilon}{2B}} f_m X_1 f_{m_1} X_2 f_{m_2} \cdots X_k f_{m_k}.
    \end{equation*}
    Then apply the base case to obtain $m_{k+1}$ such that
    \begin{equation*}
      f_{m_k} X_{k+1} \approx_{\frac{\epsilon}{2 B^k}} f_{m_k} X_{k+1} f_{m_{k+1}}.
    \end{equation*}
    Therefore,
    \begin{align*}
      f_m X_1 X_2 \cdots X_{k+1} &\approx_{\frac{\epsilon}{2}} f_m X_1 f_{m_1} X_2 f_{m_2} \cdots X_k f_{m_k} X_{k+1} \\
      &\approx_{\frac{\epsilon}{2}} f_m X_1 f_{m_1} X_2 f_{m_2} \cdots X_k f_{m_k} X_{k+1} f_{m_{k+1}},
    \end{align*}
    completing the inductive step. \qedhere
  \end{enumerate}
\end{proof}

\begin{lemma}
  \label{lem:standard-chain-approx-2}
  Let $X \in \Mul(\B \otimes \Z)$ and let $b \in \B \otimes \Z$ be arbitrary.
  For every $n \ge 1$ and every $\epsilon > 0$ and every $L' \ge 1$ and every finite collection $Y_1,\ldots,Y_l \in \Mul(\B \otimes \Z)$, there exists an $L \ge 1$ such that if $X_1,\ldots,X_n,X'_1,\ldots,X'_n \in \Mul(\B \otimes \Z)$, where for every $1 \le j \le n$,
  \begin{enumerate}
  \item either $X_j = X'_j = Y_i$ for some $1 \le i \le l$, 
  \item \label{item:Xj-X'j-equal-tails} or $X_j, X'_j \in \mathfrak{S}$ is standard, and $X_j (1-f_{L'}) = X'_j (1-f_{L'})$,
  \end{enumerate}
  then
  \begin{equation*}
    X_1 \cdots X_n (1-f_L) \approx_{\epsilon} X'_1 \cdots X'_n (1-f_L).
  \end{equation*}
\end{lemma}

\begin{proof}
  Proceed by induction on $n$.
  The base case is trivial because we may select $L = L'$, independent of $\epsilon$.

  For the inductive step, suppose that for $k \in \nat$ and $\epsilon > 0$, and suppose that we have $L''$ such that
  \begin{equation*}
    X_1 \cdots X_k (1-f_{L''}) \approx_{\frac{\epsilon}{2B}} X'_1 \cdots X'_k (1-f_{L''}).
  \end{equation*}
  By \Cref{lem:standard-chain-approx-1}, there is an $L'''$ (depending only on $\epsilon$, $k$ and $Y_1,\ldots,Y_l$) for which
  \begin{equation*}
    f_{L''} X'_{k+1} \approx_{\frac{\epsilon}{2B^k}} f_{L''} X'_{k+1} f_{L'''}
  \end{equation*}

  Then
  \begin{align*}
    &\quad\  X_1 \cdots X_k X_{k+1} \\
    &= X_1 \cdots X_k X'_{k+1} + X_1 \cdots X_k (X_{k+1} - X'_{k+1}) f_{L'} \\
    &\approx_{\frac{\epsilon}{2}} X'_1 \cdots X'_k X'_{k+1} + (X_1 \cdots X_k - X'_1 \cdots X'_k) f_{L''} X'_{k+1} \\
    &\qquad + X_1 \cdots X_k (X_{k+1} - X'_{k+1}) f_{L'} \\
    &\approx_{\frac{\epsilon}{2}} X'_1 \cdots X'_k X'_{k+1} + (X_1 \cdots X_k - X'_1 \cdots X'_k) f_{L''} X'_{k+1} f_{L'''} \\
    &\qquad + X_1 \cdots X_k (X_{k+1} - X'_{k+1}) f_{L'}.
  \end{align*}
  Choosing $L \defeq \max \set{L',L'''} + 2$, we obtain
  \begin{equation*}
    X_1 \cdots X_{k+1} (1-f_L) \approx_{\epsilon} X'_1 \cdots X'_{k+1} (1-f_L),
  \end{equation*}
  completing the inductive step. \qedhere
\end{proof}

\begin{lemma}
  \label{lem:laurent-uniform-continuity}
  Let $h$ be any Laurent polynomial and $\A$ a \cstar-algebra.
  Then $h$ is uniformly continuous on bounded subsets of $\A$.
\end{lemma}

\begin{proof}
  Note that it suffices to prove this in case the Laurent polynomial $h$ is $h(z) = z^n$ or $h(z) = \bar z^n$.
  Indeed, given $\epsilon > 0$ and a bound $M$, one may choose $\delta \defeq \frac{\epsilon}{n M^{n-1}}$, owing to the equality
  \begin{equation*}
    z^n - w^n = \sum_{k=1}^n z^{n-k} (z-w) w^{k-1},
  \end{equation*}
  which holds even for noncommutative variables $z,w$.
\end{proof}

\begin{lemma}
  \label{lem:16u}
  Let $h$ be a Laurent polynomial and let $X \in \Mul(\B \otimes \Z)$
  be contractive.
  \begin{enumerate}
  \item For every $\epsilon > 0$ there exists $N \ge 1$ such that the following holds:
    Fix contractive elements $a''_1,\ldots,a''_N \in \B \otimes \Z$ with
    \begin{equation*}
      \dist(a''_j, \her_{\B \otimes \Z}(s_j)) < \epsilon_j,
    \end{equation*}
    for all $1 \le j \le N$.

    For every $L' \ge N$, there exists an $L_1 \ge 1$
    where for all sums $\sum_j a_j$ and $\sum_j a'_j$ in
    $\mathfrak{S}$ for which $a_j = a'_j = a''_j$ for $1 \le j \le N$ and
    \begin{equation*}
      a_j = a'_j
    \end{equation*}
    for all $j \geq L'$, we have that
    \begin{equation*}
      h\left( \sum_{j=1}^{\infty} a_j X \right) (1 -f_{L_1})
      \approx_{\epsilon}
      h\left( \sum_{j=1}^{\infty} a_j' X \right) (1 - f_{L_1}).
    \end{equation*}
  \item \label{item:Jun0320213PM} For every $\epsilon > 0$, there exists $N \ge 1$ such that the following holds:
    For every $y \in \B \otimes \Z$ with $\norm{y} \le 3$ and contractive elements $a''_1,\ldots,a''_N \in \B \otimes \Z$ with
    \begin{equation*}
      \dist(a''_j, \her_{\B \otimes \Z}(s_j)) < \epsilon_j,
    \end{equation*}
    there exists an $M \geq N$ where for all sums $\sum_j a_j$ and $\sum_j a'_j$ in
    $\mathfrak{S}$ for which $a_j = a'_j = a''_j$ for all $1 \le j \le N$ and 
    \begin{equation*}
      a_j = a'_j
    \end{equation*}
    for all $j \leq M$, we have that
    \begin{equation*}
      h\left( \sum_{j=1}^{\infty} a_j X \right) y
      \approx_{\epsilon}
      h\left( \sum_{j=1}^{\infty} a'_j X \right) y.
    \end{equation*}
  \end{enumerate}
\end{lemma}

\begin{proof}
  Note that it suffices to prove this in case the Laurent polynomial $h$ is $h(z) = z^n$ or $h(z) = \bar z^n$.

  \begin{enumerate}
  \item Let $\epsilon > 0$, choose $N$ such that $\sum_{j=N+1}^{\infty} \epsilon_j < \frac{\epsilon}{n(3\norm{X})^n}$.
    Fix contractive $a''_1,\ldots,a''_N \in \B \otimes \Z$ with 
    \begin{equation*}
      \dist(a''_j, \her_{\B \otimes \Z}(s_j)) < \epsilon_j,
    \end{equation*}
    for all $1 \le j \le N$.

    Let $L' \ge N$.
    Choose $L_1$ from \Cref{lem:standard-chain-approx-2} corresponding to $2n$, $\frac{\epsilon}{2^n 3}$, $L'$ and the finite collection $A'' \defeq \sum_{j=1}^N a''_j, A''^{*}, X, X^{*} \in \Mul(\B \otimes \Z)$.

    Then consider $A \defeq \sum_j a_j, A' \defeq \sum_j a'_j \in \mathfrak{S}$ such that $a_j = a'_j = a''_j$ for $1 \le j \le N$ and $a_j = a'_j$ for all $j \ge L_1$.
    Then there exist $b_j,b_j' \in \her_{\B \otimes \Z}(s_j)$ such that $\norm{a_j - b_j}, \norm{a'_j - b'_j} < \epsilon_j$, and we may also assume $b_j = b'_j$ for $j \ge L'$
    Set $B \defeq \sum_{j=N+1}^{\infty} b_j, B' \defeq \sum_{j=N+1}^{\infty} b'_j \in \mathfrak{S}$, which are standard.
    Then
    \begin{equation*}
      \norm{AX - (A'' + B)X}, \norm{A'X - (A'' + B')X} < \frac{\epsilon}{3n(3\norm{X})^{n-1}}.
    \end{equation*}
    Consequently, since $A, A', A'' + B, A'' + B' \in \mathfrak{S}$ their norms are bounded by $3$, and hence the uniform continuity of $h$ yields
    \begin{equation*}
      h(AX) \approx_{\frac{\epsilon}{3}} h((A''+ B)X) \quad\text{and}\quad h(A'X) \approx_{\frac{\epsilon}{3}} h((A''+ B')X).
    \end{equation*}

    Expanding out $h((A'' + B)X)$ results in an expression consisting of $2^n$ terms, each of which is a word of length $2n$ in the variables $A'', B, X$ (or $A''^{*}, B^{*}, X^{*}$ if $h(z) = \bar z^n$).
    Applying \Cref{lem:standard-chain-approx-2} to each word $W$ and its corresponding word $W'$ from the expansion of $h((A'' + B')X)$, guarantees that $W(1-f_{L_1}) \approx_{\frac{\epsilon}{2^n 3}} W'(1-f_{L_1})$.
    Therefore,
    \begin{equation*}
      h((A'' + B)X) (1-f_{L_1}) \approx_{\frac{\epsilon}{3}} h((A'' + B')X) (1-f_{L_1}),
    \end{equation*}
    and consequently,
    \begin{equation*}
      h(AX) \approx_{\frac{\epsilon}{3}} h((A''+ B)X) \approx_{\frac{\epsilon}{3}} h(A'X) \approx_{\frac{\epsilon}{3}} h((A''+ B')X).
    \end{equation*}
  \item Let $\epsilon > 0$ and $y \in \B \otimes \Z$, choose $N$ such that $\sum_{j=N+1}^{\infty} \epsilon_j < \frac{\epsilon}{5n3^{n-1}\norm{X}^n \norm{y}}$.
    Fix contractive $a''_1,\ldots,a''_N \in \B \otimes \Z$ with 
    \begin{equation*}
      \dist(a''_j, \her_{\B \otimes \Z}(s_j)) < \epsilon_j,
    \end{equation*}
    for all $1 \le j \le N$.

    Choose $m \in \nat$ such that $f_m y \approx_{\frac{\epsilon}{5(3\norm{X})^n}} y$.
    Then apply \Cref{lem:standard-chain-approx-1} to $m,n,\frac{\epsilon}{2^n 10 \norm{y}}$ and the finite collection $A'' \defeq \sum_{j=1}^N a''_j, A''^{*}, X, X^{*}$ to obtain $m_1,\ldots,m_n$.
    Set $M \defeq \max \set{m_1,\ldots,m_n} + 1$.

    Then consider $A \defeq \sum_j a_j, A' \defeq \sum_j a'_j \in \mathfrak{S}$ such that $a_j = a'_j = a''_j$ for $1 \le j \le N$ and $a_j = a'_j$ for all $1 \le j \le M$.
    Then there exist $b_j,b_j' \in \her_{\B \otimes \Z}(s_j)$ such that $\norm{a_j - b_j}, \norm{a'_j - b'_j} < \epsilon_j$, and we may also assume $b_j = b'_j$ for $1 \le j \le M$.
    Set $B \defeq \sum_{j=N+1}^{\infty} b_j, B' \defeq \sum_{j=N+1}^{\infty} b'_j \in \mathfrak{S}$, which are standard.
    Then
    \begin{equation*}
      \norm{AX - (A'' + B)X}, \norm{A'X - (A'' + B')X} < \frac{\epsilon}{5n(3\norm{X})^{n-1}\norm{y}}.
    \end{equation*}
    Consequently, since $A, A', A'' + B, A'' + B' \in \mathfrak{S}$ their norms are bounded by $3$, and hence the uniform continuity of $h$ yields
    \begin{align*}
      h(AX) f_m y &\approx_{\frac{\epsilon}{5}} h((A''+ B)X) f_m y \\
      h(A'X) f_m y &\approx_{\frac{\epsilon}{5}} h((A''+ B')X) f_m y.
    \end{align*}

    Expanding out $h((A'' + B)X)$ results in an expression consisting of $2^n$ terms, each of which is a word of length $2n$ in the variables $A'', B, X$ (or $A''^{*}, B^{*}, X^{*}$ if $h(z) = \bar z^n$).
    Applying \Cref{lem:standard-chain-approx-1} to each word $W = X_1 \cdots X_n$ and its corresponding word $W' = X'_1 \cdots X'_n$ from the expansion of $h((A'' + B')X)$, we obtain
    \begin{align*}
      X_1 \cdots X_n f_m &\approx_{\frac{\epsilon}{2^n 10 \norm{y}}} f_{m_1} X_1 f_{m_2} \cdots f_{m_n} X_{m_n} f_m, \\
      X'_1 \cdots X'_n f_m &\approx_{\frac{\epsilon}{2^n 10 \norm{y}}} f_{m_1} X'_1 f_{m_2} \cdots f_{m_n} X'_{m_n}  f_m.
    \end{align*}
    By the choice of $M$ and $B,B'$, $f_M B f_M = f_M B' f_M$, and hence $f_{m_j} B f_{m_{j+1}} = f_{m_j} B' f_{m_{j+1}}$ for all $1 \le j \le n$.
    Therefore,
    \begin{equation*}
      h((A''+ B)X) f_m y \approx_{\frac{\epsilon}{5}} h((A'' + B')X) f_m y.
    \end{equation*}
    Finally,
    \begin{gather*}
      h(AX)y \approx_{\frac{\epsilon}{5}} h(AX) f_m y \approx_{\frac{\epsilon}{5}} h((A''+B)X) f_m y, \\
      h(A'X)y \approx_{\frac{\epsilon}{5}} h(A'X) f_m y \approx_{\frac{\epsilon}{5}} h((A''+B')X) f_m y,
    \end{gather*}
    and hence $h(AX)y \approx_{\epsilon} h(A'X)y$. \qedhere
  \end{enumerate}
\end{proof}

\begin{lemma}
  \label{lem:May2920215PM}
  Suppose that $X \in \Mul(\B \otimes \Z)$.

  For all $\epsilon > 0$, $n \geq 1$ and $m \geq 1$,
  there exist 
  \begin{equation*}
    1 \leq n_1 < n_2 < ... < n_m
  \end{equation*}
  and 
  \begin{equation*}
    1 \leq l_1 < l_2 < ... < l_m
  \end{equation*}
  such that if $A \in \Mul( \B \otimes \Z)$, $\norm{A} \leq 3$ and
  \begin{equation*}
    A = \sum_{j=1}^m d_j + A'
  \end{equation*}
  where
  $\norm{A'} \leq 3$ with $A' \in \her_{\Mul(\B \otimes \Z)} (1 - f_{n_m - 2})$ and
  $d_j \in \her_{\B \otimes \Z}(s_{n_j, n_{j-1}+1})$ for all $j$  ($n_0 \defeq 0$; so $s_{n_1,0} = f_{n_1}$),
  then
  \begin{equation*}
    (AX)^m f_n = A_m f_{l_m} X A_{m-1} f_{l_{m-1}} X .... A_1 f_{l_1} X f_n.
  \end{equation*}
  In the above, for $1 \leq j \leq m$,
  \begin{equation*}
    A_j \defeq \sum_{l=1}^j d_l.
  \end{equation*}
\end{lemma}

\begin{lemma}
  \label{lem:Apr1720216AM}
  Let $\{ c_k \}$ be a sequence of contractive operators 
  such that
  \begin{equation*}
    c_k \in \overline{(r_k \otimes 1) (\B \otimes \Z) (r_k \otimes 1)}
  \end{equation*}
  for all $k$, and $X \in \Mul(\B \otimes \Z)$.
  Let $n \geq 1$ be given.   
  
  Then we have the following:

  \begin{enumerate}
  \item \label{item:Apr1720216AM-part-i} 
    For every $\epsilon > 0$, there exists $L \geq 1$ and $M \geq 1$ such that
    for every $l \geq L$,
    for every series $\sum_k a_k \in \mathfrak{S}$ for which
    \begin{equation*}
      a_k = u_{l} c_k u_{l}^* \makebox{  for all  } 1 \leq k \leq M,
    \end{equation*}
    we have that
    \begin{equation*}
      f_n {h}\left(\sum_{k=1}^{\infty} a_k  X\right) f_n
      \approx_{\epsilon}
      f_n h\left(\sum_{k=1}^{\infty} u_{l} c_k u_{l}^{*} X\right) f_n.
    \end{equation*}
    
  \item \label{item:Apr1720216AM-part-ii}
    Suppose, in addition, that $\{ b_l \}$ is a sequence of elements of $\B \otimes \Z$ for which
    \begin{equation*}
      b_l \rightarrow f_n
    \end{equation*}
    in norm as $l \rightarrow \infty$.

    For every $\epsilon > 0$, there exists $L \geq 1$ and $M \geq 1$ such that
    for every $l \geq L$,
    for every series $\sum_k a_k \in \mathfrak{S}$ for which
    \begin{equation*}
      a_k = u_{l} c_k u_{l}^* \makebox{  for all  } 1 \leq k \leq M,
    \end{equation*}
    we have that
    \begin{equation*}
      b_l {h}\left(\sum_{k=1}^{\infty} a_k  X\right) b_l^*
      \approx_{\epsilon}
      b_l {h}\left(\sum_{k=1}^{\infty} u_{l} c_k u_{l}^{*} X\right) b_l^*.
    \end{equation*}
  \end{enumerate}
\end{lemma}

\begin{proof}[Sketch of proof]
  \ref{item:Apr1720216AM-part-i} We may assume that $h$ is a monomial, i.e., $h$ is the scalar multiple of either $h(z) = z^m$ or $h(z) = \overline{z}^m$
  for some $m \geq 1$.   For simplicity, let us assume that $h(z) = z^m$.

  Let $\epsilon >  0$ be given.   Let $n \geq 1$ be given.
  Plug $X$, $\frac{\epsilon}{10}$, $n$ and $m$ into \Cref{lem:May2920215PM} to get 
  \begin{equation*}
    1 \leq n_1 < .... < n_m
  \end{equation*}
  and
  \begin{equation*}
    1 \leq l_1 < .... < l_m.
  \end{equation*}

  Since $h$ is uniformly continuous on bounded subsets of $\Mul(\B \otimes \Z)$, choose
  $M > n_m$ so that 
  if $\sum_k b_k \in \mathfrak{S}$ then there exists a sequence $\{ b'_k \}$ of contractive
  elements in $\B \otimes \Z$ for which
  \begin{equation} h\left(\sum_k b_k X \right)  \approx_{\epsilon/10} h\left(\sum_k b'_k X \right)
    \label{equ:May2920215:30PM}
  \end{equation}
  and
  \begin{equation*}
    b'_k = 
    \begin{cases}
      b_k & 1 \leq k \leq M \\
      \in \her_{\B \otimes \Z}(s_k) & k \geq M + 1.
    \end{cases}
  \end{equation*}

  Note that for all $k$,
  \begin{equation*}
    \norm{ u_l (e_k \otimes 1_{\Z}) u_l^* - f_k } \rightarrow 0,
  \end{equation*}
  and thus,
  \begin{equation} dist(u_l c_k u_l^*,  \her_{\B \otimes \Z}(s_k)) \rightarrow 0
    \label{equ:May2920215:10PM}
  \end{equation}
  as $l \rightarrow \infty$.

  Note also that for all $j, l \geq 1$ and $\gamma > 0$, if
  \begin{equation*}
    u_l(e_j \otimes 1) u_l^* \approx_{\gamma} f_j
  \end{equation*}
  then
  \begin{equation} u_l (1 - (e_j \otimes 1))u_l^* \approx_{\gamma} 1 - f_j. \label{equ:May2920215:20PM}
  \end{equation}

  Finally, observe that
  for all $l$,
  \begin{equation*}
    u_l (1 - (e_{M-1} \otimes 1))u_l^* \sum_{k=M +1}^{\infty} u_l c_k u_l^* = 
    \sum_{k=M +1}^{\infty} u_l c_k u_l^*.
  \end{equation*}

  From \eqref{equ:May2920215:10PM} and \eqref{equ:May2920215:20PM}, and from the fact that 
  $h$ is uniformly continuous on bounded subsets of $\Mul(\B \otimes \K)$, 
  we can choose $L \geq 1$ so that for all $l \geq L$, there exist $c_{l,k} \in \her_{\B \otimes \Z}(s_k)$ for $1 \leq k \leq M$
  and there exist $C_l \in \her_{\Mul(\B \otimes \Z)}(1 - f_{M-1})$ such that $\norm{ c_{l,k } } \leq 1$, $\norm{ C_l } \leq 2$ for all $k$, and

  \begin{equation} 
    h\left( \sum_{k=1}^{\infty} u_l c_k u_l^*X\right) \approx_{\epsilon/10} h\left(\left(\sum_{k=1}^{M} c_{l,k} 
        + C_l \right) X\right).  \label{equ:May2920215:50PM}
  \end{equation}
  Moreover, again by uniform continuity of $h$ on bounded sets,
  increasing $L$ if necessary, we may assume that our choices of $L$ and $c_{k,l}$ ($1 \leq k \leq M$, 
  $l \geq L$) are such that for all $l \geq L$, if $\sum_k g_k \in \mathfrak{S}$ for which
  \begin{equation*}
    g_k = u_l c_k u_l^*
  \end{equation*}
  for all $1 \leq k \leq M$, then
  \begin{equation}
    h\left(\sum_{k=1}^{\infty} g_k \right) \approx_{\epsilon/10} 
    h\left( \sum_{k=1}^{M} c_{l,k} + \sum_{k=M + 1}^{\infty} g_k \right).
    \label{equ:May2920216PM}  
  \end{equation}

  Let $l \geq L$ be given.  Suppose that $\sum_k a_k \in \mathfrak{S}$ such that 
  \begin{equation*}
    a_k = u_l c_k u_l^*
  \end{equation*}
  for all $1 \leq k \leq M$.

  By \eqref{equ:May2920215:30PM},  we can find contractive $a'_k \in \her_{\B \otimes \Z}(s_k)$ for all $k \geq M + 1$ such that
  \begin{equation*}
    h\left( \left(\sum_{k=1}^{M} a_k + \sum_{k=M + 1}^{\infty} a'_k \right)X \right)
    \approx_{\epsilon/10} h\left( \sum_{k=1}^{\infty} a_k X \right).
  \end{equation*}

  So by \eqref{equ:May2920216PM},
  \begin{equation}
    \label{equ:May2920216:10PM}
    h\left( \left(\sum_{k=1}^{\infty} a_k \right)X \right) 
    \approx_{\epsilon/5} 
    h\left( \left(\sum_{k=1}^{M} c_{l,k} + \sum_{k=M + 1}^{\infty} a'_k \right)X \right). 
  \end{equation}

  For all $1\leq j \leq m$, let 
  \begin{equation*}
    d_j \defeq \sum_{k = n_{j-1} +1}^{n_j} c_{l,k}
  \end{equation*}
  and
  \begin{equation*}
    A_j \defeq \sum_{s=1}^j d_s.
  \end{equation*}

  Then, by \eqref{equ:May2920215:50PM}, \eqref{equ:May2920216:10PM}, by \Cref{lem:May2920215PM}, and since $h(z) = z^m$, 
  \begin{align*}
    h\left( \left(\sum_{k=1}^{\infty} a_k \right)X \right)f_n
    &\approx_{\frac{3\epsilon}{10}} A_m f_{l_m} X A_{m-1} f_{l_{m-1}} X ...  A_1 f_{l_1} X f_n \\
    &\approx_{\frac{\epsilon}{5}} h\left( \left(\sum_{k=1}^{\infty} u_l c_k u_l^* \right)X \right)f_n.
  \end{align*}

  Since $l \geq L$ was arbitrary, we are done.

  \ref{item:Apr1720216AM-part-ii} follows from \ref{item:Apr1720216AM-part-i} since $b_l  \rightarrow f_n$ as $l \rightarrow \infty$, and since there is a uniform bound, independent of 
  $\epsilon$, for all relevant quantities.
\end{proof}

\begin{lemma}
  \label{lem:17b}
  Let $h$ be a Laurent  polynomial and $X \in \Mul(\B \otimes \Z)$ be contractive.

  Then for every $\epsilon > 0$ there exists $N \ge 1$ so that if $a''_1, \ldots, a''_N \in \B \otimes \Z$, then for every $K \geq 1$, there exists an $L \geq 1$
  such that for every sum $\sum_j a_j \in \mathfrak{S}$ for which $a_j = a''_j$ for $1 \le j \le N$,
  \begin{equation*}
    \norm*{ f_K h\left( \sum_{j=1}^{\infty} a_j  X \right) (1 - f_L) },
    \norm*{ (1 - f_L) h\left( \sum_{j=1}^{\infty} a_j X \right) f_K }
    < \epsilon.
  \end{equation*}
\end{lemma}

\begin{proof}
  This follows easily from arguments similar to those used in the proof of \Cref{lem:16u}.
\end{proof}

\begin{lemma}
  Suppose, in addition, that $\B$ is stably finite.
  Let $a_0 \in \B \otimes \Z$ be a positive contraction.  Let $h_1, h_2, h_3 : S^1 \rightarrow [0,1]$ be continuous
  functions,
  $\delta_1 > 0$ and $0 \leq \lambda_1 < \lambda_2 < .... < \lambda_m < 1$ with $\Lambda \defeq (\lambda_1, ..., \lambda_m)$ such that
  \begin{equation*}
    h_1 h_2 = h_2
  \end{equation*}
  \begin{equation*}
    \overline{\osupp(h_3)} \subset \osupp((h_2 - \delta_1)_+)
  \end{equation*}
  and the function
  \begin{equation*}
    \lambda \mapsto \sum_{j=1}^m h_3(e^{2\pi i \lambda_j} \lambda)
  \end{equation*}
  is a full element in $C(S^1)$.

  For every $\epsilon > 0$,
  if $\widehat{h}$ is a Laurent polynomial for which
  \begin{equation*}
    | \widehat{h}(\lambda) - h_1(\lambda) | < \frac{\epsilon}{1000}
  \end{equation*}
  for all $\lambda \in S^1$ then there exists $N$ for which following holds:

  Let $a''_1,\ldots,a''_N \in \B \otimes \Z$ be contractions with
  \begin{equation*}
    \dist(a''_j,\her_{\B \otimes \Z}(s_j) ) < \epsilon_j,
  \end{equation*}
  for all $1 \le j \le N$.

  For every $L \geq N$, there exist $M > L$ and $N' \geq 1$
  such that
  for any $n \geq N'$,
  there exists $x \in \B \otimes \Z$ with
  \begin{equation*}
    \norm{ x } \leq 2 \makebox{  and  } x^* x \in \overline{s_{L, M} (\B \otimes \Z)
      s_{L, M}},
  \end{equation*}
  where if $\sum_{j=1}^{\infty} a_j \in \mathfrak{S}$
  for which $a_j = a''_j$ for $1 \le j \le N$ and
  \begin{equation*}
    a_j = u_{n} (r_j \otimes v_{\Lambda}) u_{n}^*
  \end{equation*}
  for all $L \leq j \leq M$, then
  \begin{equation*}
    x \widehat{h} \left(\sum_{j=1}^{\infty} a_j T \right) x^*
    \approx_{\epsilon} a_0.
  \end{equation*}
  \label{lem:MainTechnicalLemma}
\end{lemma}

\begin{proof}
  For simplicity, we may assume that $\lambda_1 = 0$.

  Let also $K \geq 1$ be a number such that
  \begin{equation*}
    K > \sup\{ \norm{ \widehat{h}(Y) } : Y \in \Mul(\B \otimes \Z) \makebox{ and  }
    \norm{ Y } \leq 3 \}.
  \end{equation*}
  (Note that the operators in $\mathfrak{S}$ all have norm at most $3$.)

  Plug $\widehat{h}, T, \frac{\epsilon}{100}$ into \Cref{lem:16u} to get $N \ge 1$. 
  Fix contractions $a''_1, \ldots, a''_N \in \B \otimes \Z$ such that
  \begin{equation*}
    \dist(a''_j, \her_{\B \otimes \Z}(s_j)) < \epsilon_j.
  \end{equation*}

  By \Cref{lem:16u},  choose $L_1  >  L$ such that for all $l \geq L_1 - 1$,
  \begin{equation}
    (1 - f_l) \widehat{h}\left( \sum_j a_j T \right) \approx_{\frac{\epsilon}{100}} (1 - f_l) \widehat{h}
    \left( \sum_j a'_j T \right)
    \label{equ:Apr1720217AM}
  \end{equation}
  for all $\sum_j a_j, \sum_j a'_j \in \mathfrak{S}$ for which $a_j = a'_j = a''_j$ for all $1 \le j \le N$ and
  \begin{equation*}
    a_j = a'_j \makebox{  for all  } j \geq L.
  \end{equation*}

  Recall that $\widehat{h}(\sum_{j=1}^{\infty} a_j T)$ can be expressed as a polynomial in the variables  $\sum_{j=1}^N a_j T$, $\sum_{j=N+1} a_j T$ and their adjoints.
  Hence, by \autoref{lem:standard-chain-approx-1}, increasing $L_1$ if necessary, we may assume that for all $l \ge L_1$,
  \begin{equation}
    \label{eq:Jun0320212PM}
    (1 - f_l) \widehat{h}\left( \sum_j a_j T \right) \approx_{\frac{\epsilon}{100}} (1 - f_l) \widehat{h}
    \left( \sum_j a'_j T \right)
  \end{equation}
  for all $\sum_j a_j, \sum_j a'_j \in \mathfrak{S}$ for which $a_j = a''_j$, $a'_j = s_j$ for all $1 \le j \le N$ and
  \begin{equation*}
    a_j = a'_j \makebox{  for all  } j \geq L,
  \end{equation*}
  and for which $\sum_{j=N+1}^{\infty} a_j, \sum_{j=N+1}^{\infty} a'_j \in \mathfrak{S}$ are standard.

  (Note that \eqref{equ:Apr1720217AM} and \eqref{eq:Jun0320212PM}, by uniform continuity of $\widehat{h}$, for a fixed finite number of $j > N$, we only need for $a_j,a'_j$ to be sufficiently close to $\her_{\B \otimes \Z}(s_j)$.)

  Let $C_1, ..., C_m \in \overline{(1 - f_{L_1})\Mul(\B \otimes \Z)_+(1 - f_{L_1})}$ be contractive elements such that
  \begin{equation*}
    \pi(C_j) = h_1(e^{2\pi i\lambda_j} U) \makebox{  for all  } 1 \leq j \leq m.
  \end{equation*}

  Now choose $L_2 > L_1$ for which
  \begin{equation*}
    \widehat{h}\left(\sum_{k=1}^{\infty} e^{2\pi i \lambda_j} s_k T \right) (1 - f_{L_2})
    \approx_{\frac{\epsilon}{1000}} C_j (1 - f_{L_2})
  \end{equation*}
  for all $1 \leq j \leq m$.

  Let  $A_1, ..., A_m \in \overline{(1 - e_{L_1}) \Mul(\B)_+(1 - e_{L_1})}$ and
  $T_1 \in \Mul(\B)$ be contractive elements
  such that
  \begin{equation*}
    A_j = \Phi^{-1}(C_j) \makebox{  for all  } 1 \leq j \leq m
  \end{equation*}
  and
  \begin{equation*}
    T_1 = \Phi^{-1}(T).
  \end{equation*}

  Hence,
  \begin{equation*}
    \widehat{h}\left(\sum_{k=1}^{\infty}\lambda_j r_k T_1 \right) (1 - e_{L_2})
    \approx_{\frac{\epsilon}{1000}} A_j (1 - e_{L_2})
    \makebox{  for all   } 1 \leq j \leq m.
  \end{equation*}

  By \Cref{lem:Apr720216AM,lem:Oct1920186AM},
  we can find a contractive element
  $y_1 \in \B \otimes \Z$ with $\norm{ y_1 } \leq 2$ and
  \begin{equation*}
    y_1^* y_1 \in \her_{\B \otimes \Z}((1_{\Mul(\B)} - e_{L_2 + 1}) \otimes 1_{\Z})
  \end{equation*}
  for which

  \begin{equation*}
    y_1 \left( \sum_{j=1}^m A_j \otimes z_{m,j} \right) y_1^*
    \approx_{\frac{\epsilon}{200}} a_0.
  \end{equation*}

  Hence,
  \begin{equation*}
    y_1 \left( \sum_{j=1}^m \widehat{h}\left(\sum_{k=1}^{\infty} e^{2\pi i\lambda_j} r_k T_1 \right)
      \otimes z_{m,j} \right) y_1^* \approx_{\frac{\epsilon}{60}} a_0.
  \end{equation*}

  To simplify notation, let
  \begin{equation*}
    (I) \defeq y_1 \left( \sum_{j=1}^m \widehat{h}\left(\sum_{k=1}^{\infty} e^{2\pi i\lambda_j} r_k T_1 \right)
      \otimes z_{m,j} \right) y_1^*.
  \end{equation*}

  Let
  \begin{equation*}
    \Lambda \defeq (\lambda_1, ..., \lambda_m).
  \end{equation*}
  Note that for all $1 \leq j \leq m$,
  \begin{equation*}
    v_{\Lambda} z^{1/2}_{m,j} = e^{2\pi i\lambda_j} z^{1/2}_{m,j}.
  \end{equation*}

  Hence,
  \begin{align*}
    (I) &=  y_1 \left[ \sum_{j=1}^m  (1 \otimes z^{1/2}_{m,j})
              \widehat{h} \left( \sum_{k=1}^{\infty} e^{2\pi i\lambda_j} r_k T_1  \otimes 1_{\Z}
              \right)  (1 \otimes z^{1/2}_{m,j}) \right] y_1^* \\
        &=  y_1 \left[ \sum_{j=1}^m  (1 \otimes z^{1/2}_{m,j})
              \widehat{h} \left( \sum_{k=1}^{\infty} r_k T_1  \otimes e^{2\pi i\lambda_j} 1_{\Z}
              \right)  (1 \otimes z^{1/2}_{m,j}) \right] y_1^* \\
        &= y_1 \left[ \sum_{j=1}^m  (1 \otimes z^{1/2}_{m,j})
              \widehat{h} \left( \sum_{k=1}^{\infty} r_k T_1  \otimes v_{\Lambda}
              \right)  (1 \otimes z^{1/2}_{m,j}) \right] y_1^*
  \end{align*}

  Now for all $1 \leq j \leq m$,
  $1_{\Mul(\B)} \otimes z^{1/2}_{m,j}$
  commutes with $ \widehat{h} \left( \sum_{k=1}^{\infty} r_k T_1
    \otimes v_{\Lambda}
  \right)$.

  Hence,
  defining
  \begin{equation*}
    y_2 \defeq y_1 \sum_{j=1}^m 1_{\Mul(\B)} \otimes z^{1/2}_{m,j},
  \end{equation*}
  we have that
  \begin{equation*}
    (I) = y_2 \widehat{h} \left(\sum_{k=1}^{\infty} r_k T_1 \otimes v_{\Lambda}
    \right) y_2^*.
  \end{equation*}

  Since $y_2 \in \B \otimes \Z$, we can choose $M_1 > L_2$ so that
  \begin{equation*}
    y_2 \approx_{\frac{\epsilon}{1000K}} y_2 (r_{M_1, L_1} \otimes 1_{\Z}),
  \end{equation*}
  and so that
  \begin{equation*}
    (I) \approx_{\frac{\epsilon}{200}} y_3  \widehat{h}
    \left(\sum_{k=1}^{\infty} r_k T_1 \otimes v_{\Lambda}
    \right) y_3^*,
  \end{equation*}
  where
  \begin{equation*}
    y_3 \defeq
    y_2 (r_{M_1, L_1} \otimes 1_{\Z}).
  \end{equation*}

  Noting that $r_{M_1, L_1} \otimes 1_{\Z}$ commutes with
  $1_{\Mul(\B)} \otimes v_{\Lambda}$,
  we can find
  $M > M_1$ so that
  \begin{equation}
    \label{eq:Jun03202112PM}
    y_3 \widehat{h}\left(\sum_{k=1}^{\infty} r_k T_1 \otimes v_{\Lambda}
    \right) y_3^* \approx_{\frac{\epsilon}{200}}
    y_3 \widehat{h}\left(\sum_{k=1}^{M} r_k T_1 \otimes v_{\Lambda}
    \right) y_3^*.
  \end{equation}
  Increasing $M$ if necessary, we may assume that $M$ satisfies \Cref{lem:16u}\ref{item:Jun0320213PM} (for $y = s_{M_1,L_1}$, $X = T$ and for $\frac{\epsilon}{100}$).
  

  For all $l$, by \eqref{eq:Jun03202112PM},
  \begin{align*}
    (I)
    &\approx_{\frac{\epsilon}{200}}
                                       y_3 u^*_l \widehat{h}\left( u_l(1 \otimes v_{\Lambda})
                                       u^*_l u_l \left(\sum_{k=1}^{\infty} r_k \otimes 1 \right) u^*_l u_l(T_1 \otimes 1)
                                       u_l^* \right) u_l y_3^* \\
    &\approx_{\frac{\epsilon}{100}}
                                       y_3 u^*_l \widehat{h}\left( u_l(1 \otimes v_{\Lambda})
                                       u^*_l u_l \left(\sum_{k=1}^{M} r_k \otimes 1 \right) u^*_l u_l(T_1 \otimes 1)
                                       u_l^* \right) u_l y_3^*.\\
  \end{align*}

  But as $l \rightarrow \infty$,
  \begin{equation*}
    u_l \left( \sum_{k=1}^{M} r_k \otimes 1 \right) u_l^* \rightarrow
    \sum_{k=1}^{M} s_k \makebox{   in norm,}
  \end{equation*}
  and by \Cref{lem:JiangSuMapsStrictExtension},
  \begin{equation*}
    u_l (T_1 \otimes 1) u_l^* \rightarrow T \makebox{  strictly.}
  \end{equation*}
  Since $\widehat{h}$ is uniformly continuous, we can choose $L' \geq L_2$ such that for all $l \geq L'$,
  \begin{align*}
    (I) &\approx_{\frac{\epsilon}{60}} 
                                          y_3 u^*_l \widehat{h}\left( u_l(1 \otimes v_{\Lambda})
                                          u^*_l u_l \left(\sum_{k=1}^{M} r_k \otimes 1 \right) u^*_l
          T \right) u_l y_3^*
  \end{align*}

  Hence, by the definition of (I), for all $l \geq L'$,
  \begin{equation*}
    a_0 \approx_{\frac{\epsilon}{20}}
    y_3 u^*_l \widehat{h}\left( u_l \left(\sum_{k=1}^{M}
        r_k \otimes
        v_{\Lambda} \right) u^*_l
      T \right) u_l y_3^*.
  \end{equation*}

  Note that $\dist(u_l(r_k \otimes v_{\Lambda})u_l^{*}, \her_{\B \otimes \Z}(s_k)) \to 0$ as $l \to \infty$.
  By uniform continuity of $\widehat{h}$ and by \eqref{eq:Jun0320212PM} and the remark following it, there is an $L'' > L'$ such that for all $l \ge L''$, 
  \begin{equation*}
    a_0 \approx_{\frac{3\epsilon}{50}} y_3 u_l^{*} \widehat{h} \left( \sum_{k=1}^{M} a_k T \right) u_l y_3^{*}
  \end{equation*}
  for all $\sum_{k=1}^{M} a_k \in \mathfrak{S}$ for which $a_k = a''_k$ for $1 \le k \le N$ and $a_k = u_l(r_k \otimes v_{\Lambda}) u_l^{*}$ (whose distance to $\her_{\B \otimes \Z}(s_k)$ is sufficiently small, especially very much less than $\epsilon_k$) for $N < k \le M$.

  Note that $y_3 u_l^{*} = y_2 u_l^{*} u_l (r_{M_1,L_1} \otimes 1) u_l^{*}$ and $u_l (r_{M_1,L_1} \otimes 1) u_l^{*} \to s_{M_1,L_1}$ as $l \to \infty$.
  
  Hence by \Cref{lem:16u}\ref{item:Jun0320213PM} and our choice of $M$ we have
  \begin{equation*}
    a_0 \approx_{\epsilon} x \widehat{h}\left( \sum_{k=1}^{\infty} a_k T \right) x^{*} 
  \end{equation*}
  for every $\sum_k a_k \in \mathfrak{S}$ where $a_k = a''_k$ for $1 \le k \le N$ and $a_k = u_l (r_k \otimes v_{\Lambda}) u_l^{*}$ for $L \le k \le M$, where $x = y_3 u_l^{*} s_{M_1,L_1}$ for large enough $l > L''$.
\end{proof}

The following lemma is the analogue of \cite[Lemma~4.6]{LN-2020-IEOT}.
While the proof has some similarities, there are many nontrivial
technical modifications and additions.

\begin{lemma}
  Suppose, in addition, that $\B$ is stably finite.
  Let $\A$ be a unital separable \cstar-algebra and
  $\phi : \A \rightarrow \Mul(\B \otimes \Z)$  a unital *-homomorphism.

  Then there exist a subsequence $\{ n_l \}$ of the positive integers
  and a sequence $\{ v_l \}$ of unitaries in $\Z$ such that
  $\pi(\sum_{l=1}^{\infty} u_{n_l}(r_l \otimes v_l) u_{n_l}^*)$
  is a unitary in the connected component of $1$
  in $(\pi\circ \phi(\A))'$, and
  the unitary
  $\pi(\sum_{l=1}^{\infty} u_{n_l} (r_l \otimes v_l) u_{n_l}^*)U$
  is a strongly full element
  of $\C(\B \otimes \Z)$.
  \label{lem:PreTechnicalLemma}
\end{lemma}

\begin{proof}
  For every $n \geq 1$,
  let $h_{n, 1, j}, h_{n, 2, j}, h_{n, 3, j} : S^1 \rightarrow [0,1]$
  be continuous functions, $\lambda_{n,j} \in [0,1)$ (for $1 \leq j \leq n$),
  and
  $\delta_n > 0$ be such that
  \begin{equation*}
    \sum_{j=1}^n h_{n, 3, j}
  \end{equation*}
  is a full element of  $C(S^1)$.
  Hence for every $k$,
  \begin{equation*}
    \sum_{j=1}^n h_{n, 3, j}(e^{2\pi i \lambda_{n,k}} \lambda)
  \end{equation*}
  is a full element of $C(S^1)$.
  \begin{equation*}
    \overline{\osupp(h_{n, 3, j})} \subset \osupp((h_{n, 2, j} - \delta_n)_+),
  \end{equation*}
  \begin{equation*}
    h_{n, 1, j} h_{n, 2, j} = h_{n, 2, j},
  \end{equation*}
  \begin{equation*}
    h_{n,1,j}(\lambda) = h_{n, 1, 1}(e^{2\pi i\lambda_{n,j}} \lambda) \makebox{  for all  } \lambda \in S^1,
  \end{equation*}
  \begin{equation*}
    h_{n, 3, j}(\lambda) = h_{n, 3, 1}(e^{2\pi i\lambda_{n,j}} \lambda)  \makebox{  for all  } \lambda \in S^1,
  \end{equation*}
  there are infinitely many $1 \leq j' \leq n'$ for which
  \begin{equation*}
    h_{n, 1, j} = h_{n', 1, j'}
  \end{equation*}
  for all $1 \leq j \leq n$, and
  \begin{equation*}
    \liminf_{n \rightarrow \infty} \diam(\osupp(h_{n, 1, 1})) = 0.
  \end{equation*}

  In addition, we impose
  the following conditions on the scalars $\lambda_{n,j}$:

  \begin{equation*}
    | \lambda_{n, j+1} - \lambda_{n,j} | < \frac{10 \pi}{n},
  \end{equation*}
  for all $1 \leq j \leq n - 1$.

  And we may assume that for all $n \geq 1$,
  \begin{equation*}
    0 = \lambda_{n,1} < \lambda_{n, 2} < ... < \lambda_{n,n} < 1.
  \end{equation*}

  We denote the above statements by ``(+)".

  Let $\{ g_k \}_{k=1}^{\infty}$ be the sequence of continuous functions
  from $S^1$ to $[0,1]$ and $\{ \Lambda_k \}_{k=1}^{\infty}$ be the sequence
  of vectors (of varying dimensions and with entries in $S^1$) given by
  \begin{equation*}
    g_k = h_{n, 1, j}
  \end{equation*}
  and
  \begin{equation*}
    \Lambda_k \defeq (\lambda_{n,1}, ..., \lambda_{n,n} )
  \end{equation*}
  when
  \begin{equation*}
    k = \frac{(n-1)n}{2} + j \makebox{ and  } 1 \leq j \leq n.
  \end{equation*}
  Note that each term in $\{ g_k \}$ reappears in the sequence infinitely many times.
  Also, for each $k$, we will be considering the unitary
  $v_{\Lambda_k} \in \dimdrop
  \subset \Z$.  (Recall that $v_{\Lambda_k}$ is norm-path-connected to $1$
  via a continuous path of unitaries with length at most $3 \pi$.)

  Let $\{ c_k \}$ be a sequence of pairwise orthogonal contractions in
  $(\B \otimes \Z)_+$ such
  for every subsequence $\{ k_l \}$ of the positive integers, there exists a
  contraction $Y \in \Mul(\B \otimes \Z)$ for which
  \begin{equation*}
    Y\left(\sum_{l=1}^{\infty} c_{k_l} \right) Y^* = 1_{\Mul(\B)}
  \end{equation*}
  where the sum converges strictly.

  Let $\{ \epsilon_{k,l} \}$ be a (decreasing in $k+l$)
  biinfinite sequence in $(0,1)$ such that
  \begin{equation*}
    \sum_{1 \leq k, l < \infty} \epsilon_{k,l} < \infty.
  \end{equation*}
  We may assume that $\epsilon_{k,l} = \epsilon_{l,k}$ for all $k, l$.

  Let $\{ \epsilon_l \}_{l=1}^{\infty}$ be the strictly
  decreasing sequence in $(0,1)$
  from the definition of $\mathfrak{S}$.
  (This is from the fixed notation before \Cref{lem:Apr720216AM}.
  Recall that $\sum_{l=1}^{\infty} \epsilon_l < 1$.)

  Let $S_0 \subset \phi(\A)$ be a countable dense set in the closed unit
  ball of $\phi(\A)_+$.  Say that we have enumerated $S_0$ as
  $S_0 = \{ T_k' \}_{k=1}^{\infty}$.  Recall that by our sectional
  notational conventions (established
  after the proof of \Cref{lem:Apr120213AM}), we have $N_k'$ obtained by plugging $\set{T'_1,\ldots,T'_k}$ into \Cref{lem:Apr120213AM},
  and the sequence $\{ e_n \}$
  quasicentralizes $S \defeq \Phi^{-1}(S_0) = \{ \Phi^{-1}(T_k') \}$
  as in \Cref{lem:QuasicentralAU}.
  Notice that we can increase $N'_k$ and still satisfy the conditions of \Cref{lem:Apr120213AM}.

  By an inductive construction on $k$ (a variable ranging over the positive
  integers), we construct subsequences $\{ n_l \}_{l=1}^{\infty}$,
  $\{ L_k \}_{k=1}^{\infty}$,
  $\{ M_k \}_{k=1}^{\infty}$ and  $\{ N_k \}_{k=1}$ of the integers,
  a sequence $\{ x_k \}_{k=1}^{\infty}$ of elements of $\B \otimes \Z$ with
  at most norm $2$, a sequence
  $\{ \widehat{g}_k \}_{k=1}^{\infty}$ of Laurent polynomials,
  and a sequence $\{ v_l \}_{l=1}^{\infty}$ of unitaries
  in $\dimdrop \subset \Z$.  (The relationship between the variable $l$, in
  $\{ n_l \}$ and $\{ v_l \}$, and the induction variable
  $k$ will be explained below).
  In this construction, we apply (for some lemmas, repeatedly)
  \Cref{lem:AUUnitary,lem:Apr120213AM,lem:17b,lem:MainTechnicalLemma}.  The inductive construction (in $k$) would
  then obtain the following statements:

  \begin{enumerate}
  \item $N_{k-1} < L_k < M_k < N_k' < N_k$ for all $k \geq 1$.
  \item For all $k \geq 1$ and $L_k \leq l \leq M_k$,
    $v_l = v_{\Lambda_k}$. \label{e1}
  \item For all $k \geq 1$, $v_{N_k} = v_{N_k+1} = 1_{\Z}$. \label{e2}
  \item For all $k \geq 1$ and
    $L_k \leq l \leq L_{k+1}$,
    $\norm{ v_{l+1} - v_l } < \epsilon_k$. \label{e3}
  \item For all $k \geq 1$ and for all $N_{k-1} + 1 \leq l \leq N_k$,
    $u_{n_l} = u_{n_{N_{k}}}$. \label{e4}
  \item For all $l \geq 1$,
    $\dist(u_{n_l}(r_l \otimes 1_{\Z}) u_{n_l}^*, s_l)
    < \epsilon_l$. \label{e5}
  \item For all $k \geq 1$, for all $1 \leq j \leq k$,
    and for any contractions $z_l$ for $N_{k-1} + 1 \leq l \leq N_k$,
    \begin{equation*}
      \left[ u_{n_{N_k}} \left( \sum_{l= N_{k-1} + 1}^{N_k}
          (r_l \otimes z_l) \right) u_{n_{N_k}}^*, T_j'
      \right] \approx_{\epsilon_k} 0.
    \end{equation*} \label{e6}
  \item For all $l \geq 1$, each $v_l$ is path-connected to $1_{\Z}$ via
    a norm-continuous path $\{ v_l(t) \}_{t \in [0,1]}$ of unitaries in $\dimdrop$ with length at
    most $3 \pi$. (So $v_l(0) = v_l$ and $v_l(1) = 1_{\Z}$.)
    Moreover, we can choose these paths so that
    for all $k \geq 1$,
    $v_{N_k}(t) = v_{N_k +1}(t) = 1_{\Z}$ for all $t \in [0,1]$. \label{e7}
  \item For all $k_1, k_2 \geq 1$,
    \begin{equation*}
      \widehat{g}_{k_1} = \widehat{g}_{k_2} \makebox{  if and only if  } g_{k_1} = g_{k_2}.
    \end{equation*}
    Moreover,
    \begin{equation*}
      \max_{\lambda \in S^1} |g_k(\lambda) - \widehat{g}_k(\lambda)| <
      \frac{\epsilon_{k_0}}{10000},
    \end{equation*}
    where $k_0$ is the least integer for which $g_k = g_{k_0}$.\label{e10}
  \item For all $k \geq 1$,
    $x_k^* x_k \in \her_{\B \otimes \Z}(s_{L_k, M_k})$,
    $x_k x_k^* \in \her_{\B \otimes \Z}(c_k)$
    and
    \begin{equation*}
      x_k \widehat{g}_k\left(\sum_{j=1}^{\infty} u_{n_l}(r_l \otimes v_l)
        u_{n_l}^* T
      \right) x_k^* \approx_{\frac{\epsilon_{k_0}}{10}} c_k
    \end{equation*}
    where $k_0$ is the least integer for which $g_k = g_{k_0}$.  \label{e8}
  \item For all $k \geq 1$, for all $1 \leq l \leq k-1$,
    \begin{equation*}
      x_k \widehat{g}_k\left(\sum_{j=1}^{\infty}
        u_{n_l}(r_l \otimes v_l) u_{n_l}^* T
      \right) x_l^* \approx_{\epsilon_{k,l}} 0
    \end{equation*}
    and
    \begin{equation*}
      x_l \widehat{g}_k\left(\sum_{j=1}^{\infty} u_{n_l}(r_l \otimes v_l)
        u_{n_l}^* T
      \right) x_k^* \approx_{\epsilon_{k,l}} 0
    \end{equation*} \label{e9}
  \end{enumerate}
  We denote the above statements by ``(*)".

  By \Cref{lem:AUUnitary} and by (*) statements
  \ref{e1}, \ref{e2}, \ref{e3}, \ref{e4} and \ref{e5},
  the sum
  \begin{equation*}
    W \defeq \sum_{l=1}^{\infty} u_{n_l}(r_l \otimes v_l) u_{n_l}^*
  \end{equation*}
  converges in the strict topology on $\Mul(\B \otimes \Z)$,
  and also, $\pi(W)$ is a unitary in $\C(\B \otimes \Z)$.

  By \Cref{lem:Apr120213AM} and by (*) statements
  \ref{e6}, \ref{e7} and \ref{e5},
  $\pi(W) \in \pi(\A)'$, and $\pi(W)$ is path-connected to
  $1_{\C(\B \otimes \Z)}$ via a norm-continuous path of unitaries
  in $\pi(\A)'$.

  Let $k \geq 1$ be given.
  We will now show that $g_k(\pi(W) U)$ is full
  in $\C(\B \otimes \Z)$.
  Let $\frac{1}{10} > \epsilon > 0$ be given.
  Since each term of the sequence $\{g_l\}_{l=1}^{\infty}$ is repeated infinitely many times, we may assume that $k$ is large enough so that

  \begin{equation}
    2\sum_{l \geq k,  n \geq 1} \epsilon_{l, n}
    < \frac{\epsilon}{10}.
    \label{equ:epsilon_ks}
  \end{equation}

  Let $k_0$ be the first integer for which such that $g_{k_0} = g_k$.
  Let $\{ k_j \}_{j=1}^{\infty}$ be a subsequence of the positive
  integers such that
  \begin{equation*}
    k_1 = k
  \end{equation*}
  and
  \begin{equation*}
    g_{k_j} = g_k \makebox{  for all  } j \geq 1.
  \end{equation*}

  For all $j < s$,
  \begin{align*}
    x_{k_s} \widehat{g}_k( W T) x_{k_s}^{*}
    &= x_{k_s} \widehat{g}_{k_s} (WT) x_{k_s}^{*} \\
    &\approx_{\frac{\epsilon_{k_0}}{10}}  c_{k_s} \makebox{  (by (*) statement \ref{e8})}
  \end{align*}
  and
  \begin{align*}
    x_{k_j} \widehat{g}_k( W T) x_{k_s}^{*}
    &= x_{k_j} \widehat{g}_{k_s} (WT) x_{k_s}^{*} \\
    &\approx_{\epsilon_{j,s}} 0 \makebox{  (by (*) statement \ref{e9}).}
  \end{align*}
  Similarly,
  \begin{equation*}
    x_{k_s} \widehat{g}_k( W T) x_{k_j}^{*} \approx_{\epsilon_{j,s}}
    0.
  \end{equation*}

  By (*) statement \ref{e8} and since $\norm{x_n} \le 2$ for every $n$, and since $x_n^{*} x_m = 0$ for every $n \not= m$,
  \begin{equation*}
    X \defeq \sum_{j=1}^{\infty} x_{k_j}
  \end{equation*}
  converges strictly to an element of $\Mul(\B \otimes \Z)$ with
  norm at most $2$.

  Hence, by \eqref{equ:epsilon_ks},
  \begin{align*}
    X \widehat{g}_k(WT) X^*
    &= \sum_{j=1}^{\infty} x_{k_j} \widehat{g}_k(WT) x_{k_j}^*
          + \sum_{j \neq s} x_{k_j} \widehat{g}_k (WT) x_{k_s}^*\\
    &\approx_{\frac{\epsilon_{k_0} + \epsilon}{10}}  \sum_{j=1}^{\infty} c_{k_j}.
  \end{align*}

  But by the definition of $\{ c_l \}_{l=1}^{\infty}$, we can find
  a contraction $Y \in \Mul(\B \otimes \Z)$ for which
  \begin{equation*}
    Y \left( \sum_{j=1}^{\infty} c_{k_j} \right) Y^* = 1_{\Mul(\B \otimes \Z)}.
  \end{equation*}

  Hence,
  \begin{equation*}
    Y X \widehat{g}_k(W T) X^* Y^* \approx_{\frac{\epsilon_{k_0} + \epsilon}{10}}
    1_{\Mul(\B \otimes \Z)}.
  \end{equation*}

  Hence,
  \begin{equation*}
    \pi(YX) \widehat{g}_k(\pi(W) U) \pi(X^* Y^*)
    \approx_{\frac{\epsilon_{k_0} + \epsilon}{10}}
    1_{\C(\B \otimes \Z)}.
  \end{equation*}

  Therefore, since $\norm{ Y X } \leq 2$, and by (*) statement
  \ref{e10},
  \begin{equation*}
    \pi(YX) g_k(\pi(W) U) \pi(X^* Y^*)
    \approx_{\frac{\epsilon_{k_0} + \epsilon}{2}} 1_{\C(\B \otimes \Z)}.
  \end{equation*}
  Since $\epsilon_{k_0}, \epsilon  < 1$,
  $g_k(\pi(W)U)$ is full in $\C(\B \otimes \Z)$.
  Since $k$ was arbitrary, we have shown that for all $k \geq 1$,
  $g_k(\pi(W)U)$ is full in $\C(\B \otimes \Z)$.

  Now, by the definition of the sequence
  $\{ g_k \}$, we claim that the unitary $\pi(W)U$
  is a strongly full element of $\C(\B)$.

  To see this, note that every nonnegative continuous function $f \in C(S^1)$ has some $g_k$ which is in the ideal generated by $f$.
  Indeed, there is some arc of positive width $\eta$ centered at $s \in S^1$ on which $f$ is greater than some $\zeta > 0$.
  Since $\liminf_{n \rightarrow \infty} \max_{1 \le j \le n}
  \diam (\osupp (h_{n,1,j})) = 0$, there is some $n$ such that the maximum of these diameters is less than $\frac{\eta}{3}$.
  Moreover, since $\sum_{j=1}^n h_{n,3,j}$ is full in $C(S^1)$, there is some $1 \le j \le n$ such that $h_{n,1,j}(s) \not= 0$.
  Then, because
  \begin{equation*}
    \diam(\osupp(h_{n,1,j})) < \frac{\eta}{3},
  \end{equation*}
  the support of
  $h_{n,1,j}$ is entirely contained within the arc on which $f \ge \zeta > 0$.
  Therefore $h_{n,1,j}$ is in the ideal generated by $f$.
  Finally, by the definition of $\{ g_k \}_{l=1}^{\infty}$, there is some $k$
  for which $g_k = h_{n,1,j}$ (in fact, there are infinitely many such $k$).
\end{proof}

Having laid the groundwork, we are finally in a position to prove our main theorem concerning the $\ko$-injectivity of the Paschke dual algebra $\paschkedual{\A}{\B}$.\\

\textbf{For the convenience of the reader, we state the full assumptions in our main theorems below, which were standing assumptions starting in the present Subsection~\ref{sec:technical-lemmas}.}\\

\begin{theorem}
  \label{thm:k1-injectivity-simple-nuclear-strict-comparison}
  Suppose that $\A,\B$ are separable simple \cstar-algebras with $\A$ unital and nuclear, and $\B$ stable and $\Z$-stable.

  Then $\paschkedual{\A}{\B}$ is $\ko$-injective.
  Moreover, for each $n \in \nat$, the map
  \begin{equation*}
    U(\mat_n \otimes \paschkedual{\A}{\B})/U(\mat_n \otimes \paschkedual{\A}{\B})_0 \to U(\mat_{2n} \otimes \paschkedual{\A}{\B})/U(\mat_{2n} \otimes \paschkedual{\A}{\B})_0
  \end{equation*}
  given by
  \begin{equation*} [u] \mapsto [u \oplus 1] \end{equation*}
  is injective.
\end{theorem}

\begin{proof}
  Since $\B$ is $\Z$-stable, it is either purely infinite or stably finite.
  The case when $\B$ is purely infinite follows from \cite[Theorem~2.5]{LN-2020-IEOT}.
  So we may assume that $\B$ is stably finite.
  Hence $\B$ also has the corona factorization property since it is $\Z$-stable.
  Then the conclusion follows directly from \Cref{lem:PreTechnicalLemma} and \Cref{thm:k1-injectivity-generic}.
\end{proof}

Now we obtain our primary generalization of the BDF essential codimension result.

\maintheorem

\begin{proof}
  Since $\B$ has strict comparison, $\B$ has the corona factorization property.
  Hence, $\phi$ and $\psi$ are both absorbing extensions.

  By \Cref{thm:k1-injectivity-simple-nuclear-strict-comparison}, $\A_{\B}$ is $\ko$-injective.  Hence, the
  result follows from \Cref{thm:MainUniqueness}.
\end{proof}

\begin{remark}
  We note that \cite[Theorem~2.5]{LN-2020-IEOT} applies whenever $\B$ is simple purely infinite (or $\B = \K$), and in this case $\Z$-stability of $\B$ is not required to obtain the conclusion of \Cref{thm:main-theorem}.
\end{remark}

\subsection{Removing Jiang--Su stability}

From this subsection on, we will not longer follow the standing assumptions and notation that were made in Subsection~\ref{sec:technical-lemmas}.

$\Z$-stability is a natural regularity assumption for nonelementary simple
nuclear \cstar-algebras.  Among other things, it is conjectured that for any
separable simple nuclear \cstar-algebra, having $\Z$-stability is equivalent to
having strict comparison for positive elements.  Unfortunately, $\Z$-stability
is no longer prominent outside of the nuclear case.  For instance,
$C^*_r(\mathbb{F}_{\infty})$ is an example of a simple (nonnuclear) \cstar-algebra
with strict comparison that is not $\Z$-stable.

Our present techniques allow for the replacement of $\Z$-stability with
strict comparison, but with the
additional restriction of finitely many extreme traces.  This is still
interesting as $C^*_r(\mathbb{F}_{\infty})$ and a number of similar \cstar-algebras
have a unique tracial state.

\begin{lemma}
  Let $\B$ be a nonunital separable simple \cstar-algebra.

  Suppose that
  $\{ e_n \}_{n=1}^{\infty}$ is an approximate unit for
  $\B$ such that
  \begin{equation*}
    e_{n+1} e_n = e_n \makebox{  for all  } n\geq 1.
  \end{equation*}

  Suppose that $\{ \alpha_n \}_{n=1}^{\infty}$ is a sequence in
  $S^1$ such that for all $\epsilon > 0$, there exists an $N \geq 1$
  for which
  \begin{equation*}
    | \alpha_{n+1} - \alpha_n | < \epsilon  \makebox{  for all  }
    n \geq N.
  \end{equation*}

  Then $\pi\left( \sum_{n=1}^{\infty} \alpha_n (e_n - e_{n-1}) \right)$
  is a unitary.
  (We use the convention $e_0 \defeq 0$.)
  \label{lem:Apr2720216AM}
\end{lemma}

\begin{proof}
  The proof is similar to (in fact easier than) that of \Cref{lem:AUUnitary}.
  However, one can also just apply \Cref{lem:AUUnitary} by taking $\E = \complex$ and $u_k = 1$ for all $k$.
\end{proof}

\begin{lemma}
  Let $\B$ be a separable simple stable stably finite \cstar-algebra with strict
  comparison of positive elements such that
  $T(\B)$ has finitely many extreme points, and let $\A$ be a unital separable
  nuclear \cstar-algebra.
  Let $\phi : \A \rightarrow \Mul(\B)$ be a unital *-homomorphism and
  let $S \defeq \{ T_k : k \geq 1\}$ be a countable dense subset of
  the closed unit ball of $\phi(\A)$.

  Let $\{ e_n \}$ be an approximate unit for $\B$ which quasicentralizes
  $\phi(\A)$ so that
  \begin{equation*}
    e_{n+1} e_n = e_n \makebox{  for all  } n \geq 1
  \end{equation*}
  and
  \begin{equation*}
    \norm{ e_n T_j - T_j e_n } < \frac{1}{2^n} \makebox{  for all  } 1 \leq
    j \leq n.
  \end{equation*}

  Let $U \in \pi \circ \phi(\A)'$ be a unitary.  Then there exists a sequence
  $\{ \alpha_n \}$ in $S^1$ such that
  $\pi\left( \sum_{n=1}^{\infty} \alpha_n(e_n - e_{n-1}) \right)$
  is a unitary in
  the connected component of $1$ in $\pi \circ \phi(\A)'$, and
  the unitary $\pi\left(\sum_{n=1}^{\infty}
    \alpha_n (e_n - e_{n-1})\right) U$ is a strongly full element of
  $\C(\B)$.
  \label{lem:Apr2720217AM}
\end{lemma}

\begin{proof}
  The proof is exactly the same as that of \cite[Lemmas~4.6 and 4.7]{LN-2020-IEOT}
   (and hence, we need to replicate \cite[Lemmas~4.4 and
  4.5]{LN-2020-IEOT}).  The main difference, is that, we now do not need $\{ e_n \}$ to
  be a sequence of projections and we do not need $\phi$ to have a special
  form, since we can use (our paper) \Cref{lem:Apr2720216AM,lem:QuasicentralAU} to ensure that
  $\pi\left(\sum_{n=1}^{\infty} \alpha_n (e_n - e_{n-1})\right)$ is
  a unitary in $\pi \circ \phi(\A)'$.  E.g., thus,
  in \cite[Lemma~4.6]{LN-2020-IEOT}
  (and Lemmas 4.4 and 4.5), we replace the (pairwise orthogonal)
  projections $p_j$ with the (not necessarily pairwise orthogonal) positive
  elements $e_j - e_{j-1}$.

  We also make other obvious modifications.  E.g., we replace
  the projection $e_{L_1}^{\perp}$ in \cite[Lemma~4.4]{LN-2020-IEOT} with
  (our paper's) the positive element $1 - e_{L_1}$, and we replace
  the projection
  $p$ in \cite[Lemma~4.5]{LN-2020-IEOT} with a positive element.  With these and other
  trivial and straightforward modifications, the proofs all follow through
  smoothly, because of the considerations of the first paragraph.
\end{proof}

\begin{theorem}
  Suppose that $\A,\B$ are separable simple \cstar-algebras with $\A$ unital and nuclear, $\B$ stable and stably finite and having strict comparison of positive elements, and
  $T(\B)$ having finitely many extreme points.

  Then $\paschkedual{\A}{\B}$ is $\ko$-injective.
  Moreover, for each $n \in \nat$, the map
  \begin{equation*}
    U(\mat_n \otimes \paschkedual{\A}{\B})/U(\mat_n \otimes \paschkedual{\A}{\B})_0 \to U(\mat_{2n} \otimes \paschkedual{\A}{\B})/U(\mat_{2n} \otimes \paschkedual{\A}{\B})_0
  \end{equation*}
  given by
  \begin{equation*} [u] \mapsto [u \oplus 1] \end{equation*}
  is injective.
  \label{thm:Apr2720218AM}
\end{theorem}

\begin{proof}
  The proof is exactly the same as that of 
  \Cref{thm:k1-injectivity-simple-nuclear-strict-comparison},
  except that we replace \Cref{lem:PreTechnicalLemma}
  with \Cref{lem:Apr2720217AM}.
\end{proof}

\begin{theorem}
  \label{thm:uniqueness-finite-extreme-boundary}
  Let $\A$, $\B$ be separable simple \cstar-algebras, with $\A$ unital and nuclear,
  $\B$ stable and stably finite and having strict comparison for positive elements, and $T(\B)$
  having finitely many extreme points.

  Let $\phi, \psi : \A \rightarrow \Mul(\B)$ be two unital full trivial extensions with
  $\phi(a) - \psi(a) \in \B$ for all $a \in \A$.

  Then $[\phi, \psi] = 0$ in $\kk(\A, \B)$ if and only if $\phi \approxeq \psi$.
\end{theorem}

\begin{proof}
  The proof is exactly the same as that of \Cref{thm:main-theorem},
  except that \Cref{thm:k1-injectivity-simple-nuclear-strict-comparison}
  is replaced with \Cref{thm:Apr2720218AM}.
\end{proof}

\section*{Appendix}

\begin{lemma}
  If $\B$ is a $\sigma$-unital
  stable \cstar-algebra, then $\C(\B)$ is $\ko$-injective.
  \label{lem:C(B)Injective}
\end{lemma}

\begin{proof}
  Let $p, q \in \C(\B)$ be full properly infinite projections.
  Hence, let $x \in \C(\B)$ be such that
  \begin{equation*}
    x p x^* = 1_{\C(\B)}.
  \end{equation*}
  Let $A \in \Mul(\B)_+$ such that $\pi(A) = p$.
  Since $\B$ is stable, there exists $Y \in \Mul(\B)$ such that
  \begin{equation*}
    Y A Y^* = 1_{\Mul(\B)}.
  \end{equation*}

  Hence, there exists a projection $R \in \overline{A \Mul(\B) A}$ for which
  \begin{equation*}
    R \sim 1 \makebox{  and  } 1 - R \sim 1.
  \end{equation*}

  Similarly, there exists $A_1 \in \Mul(\B)_+$
  with
  \begin{equation*}
    \pi(A_1) = q
  \end{equation*}
  and there exists a projection $S \in \overline{A_1 \Mul(\B) A_1}$ such that
  \begin{equation*}
    S \sim 1 \makebox{  and  } 1 - S \sim 1.
  \end{equation*}

  Hence, there exists a unitary $U \in \Mul(\B)$ for which
  \begin{equation*}
    U R U^* = S.
  \end{equation*}
  Since $\B$ is stable, $\mathcal{U}(\Mul(\B))$ is (norm) contractible, so
  \begin{equation*}
    R \sim_h S.
  \end{equation*}
  Then
  \begin{equation*}
    r \defeq \pi(R) \makebox{  and  } s \defeq \pi(S)
  \end{equation*}
  are full properly infinite
  projections in $\C(\B)$ with $r \leq p$, $s \leq q$, and
  \begin{equation*}
    r \sim_h s.
  \end{equation*}
  Since $p, q$ were arbitrary, it follows, by \cite[Proposition~5.1]{BRR-2008-JNG}, that $\C(\B)$ is $\ko$-injective.
\end{proof}

\bibliographystyle{amsalpha}
\bibliography{references.bib}

\newcommand{\etalchar}[1]{$^{#1}$}
\providecommand{\bysame}{\leavevmode\hbox to3em{\hrulefill}\thinspace}
\providecommand{\MR}{\relax\ifhmode\unskip\space\fi MR }
\providecommand{\MRhref}[2]{%
  \href{http://www.ams.org/mathscinet-getitem?mr=#1}{#2}
}
\providecommand{\href}[2]{#2}
\begin{thebibliography}{{R{\o}}04}

\bibitem[Arv07]{Arv-2007-PNASU}
William Arveson, \emph{Diagonals of normal operators with finite spectrum},
  Proc. Natl. Acad. Sci. USA \textbf{104} (2007), no.~4, 1152--1158 (English).

\bibitem[BCP{\etalchar{+}}06]{BCP+-2006-Agatoeo}
Moulay-Tahar Benameur, Alan~L. Carey, John Phillips, Adam Rennie, Fyodor~A.
  Sukochev, and Krzysztof~P. Wojciechowski, \emph{An analytic approach to
  spectral flow in von {N}eumann algebras}, Analysis, geometry and topology of
  elliptic operators (Matthias Lesch, Bernhelm Boo-Bavnbek, Slawomir Klimek,
  and Weiping Zhang, eds.), World Sci. Publ., Hackensack, NJ, 2006,
  pp.~297--352. \MR{2246773}

\bibitem[BDF73]{BDF-1973-PoaCoOT}
Lawrence~G. Brown, Ronald~George Douglas, and Peter~Arthur Fillmore,
  \emph{Unitary equivalence modulo the compact operators and extensions of
  {$C^{\ast} $}-algebras}, {Proceedings of a Conference on Operator Theory}
  (Peter~Arthur Fillmore, ed.), Lecture Notes in Mathematics, vol. 345,
  Springer, Berlin, 1973, pp.~58--128. \MR{0380478 (52 \#1378)}

\bibitem[BJ15]{BJ-2015-TAMS}
Marcin {Bownik} and John {Jasper}, \emph{{The Schur-Horn theorem for operators
  with finite spectrum.}}, {Trans. Am. Math. Soc.} \textbf{367} (2015), no.~7,
  5099--5140 (English).

\bibitem[BL12]{BL-2012-CJM}
Lawrence~G. Brown and Hyun~Ho Lee, \emph{Homotopy classification of projections
  in the corona algebra of a non-simple {$C^*$}-algebra}, Canad. J. Math.
  \textbf{64} (2012), no.~4, 755--777. \MR{2957229}

\bibitem[Bla98]{Bla-1998}
Bruce Blackadar, \emph{{$K$}-theory for operator algebras}, second ed.,
  Mathematical Sciences Research Institute Publications, vol.~5, Cambridge
  University Press, Cambridge, 1998. \MR{1656031}

\bibitem[BRR08]{BRR-2008-JNG}
Etienne Blanchard, Randi Rohde, and Mikael R\o{r}dam, \emph{Properly infinite
  {$C(X)$}-algebras and {$K_1$}-injectivity}, J. Noncommut. Geom. \textbf{2}
  (2008), no.~3, 263--282. \MR{2411419}

\bibitem[Dav96]{Dav-1996}
Kenneth~R. Davidson, \emph{{$C^*$}-algebras by example}, Fields Institute
  Monographs, vol.~6, American Mathematical Society, Providence, RI, 1996.
  \MR{1402012}

\bibitem[DE01]{DE-2001-KT}
Marius Dadarlat and S\o{r}en Eilers, \emph{Asymptotic unitary equivalence in
  {$KK$}-theory}, $K$-Theory \textbf{23} (2001), no.~4, 305--322. \MR{1860859}

\bibitem[EK01]{EK-2001-PJM}
George~A. Elliott and Dan Kucerovsky, \emph{An abstract
  {V}oiculescu-{B}rown-{D}ouglas-{F}illmore absorption theorem}, Pacific J.
  Math. \textbf{198} (2001), no.~2, 385--409. \MR{1835515}

\bibitem[GJS00]{GJS-2000-CMB}
Guihua {Gong}, Xinhui {Jiang}, and Hongbing {Su}, \emph{{Obstructions to
  \({\mathcal Z}\)-stability for unital simple \(C^*\)-algebras}}, {Can. Math.
  Bull.} \textbf{43} (2000), no.~4, 418--426 (English).

\bibitem[GN19]{GN-2019-AM}
Thierry {Giordano} and Ping~W. {Ng}, \emph{{A relative bicommutant theorem: the
  stable case of Pedersen's question}}, {Adv. Math.} \textbf{342} (2019), 1--13
  (English).

\bibitem[Hig87]{Hig-1987-PJM}
Nigel Higson, \emph{A characterization of {$KK$}-theory}, Pacific J. Math.
  \textbf{126} (1987), no.~2, 253--276. \MR{869779}

\bibitem[Jas13]{Jas-2013-JFA}
John Jasper, \emph{{The Schur--Horn theorem for operators with three point
  spectrum}}, J. Funct. Anal. \textbf{265} (2013), no.~8, 1494--1521.
  \MR{3079227}

\bibitem[JS99]{JS-1999-AJM}
Xinhui {Jiang} and Hongbing {Su}, \emph{{On a simple unital projectionless
  \(C^*\)-algebra}}, {Am. J. Math.} \textbf{121} (1999), no.~2, 359--413
  (English).

\bibitem[JT91]{JT-1991}
Kjeld~Knudsen Jensen and Klaus Thomsen, \emph{Elements of {$KK$}-theory},
  Mathematics: Theory \& Applications, Birkh\"auser Boston, Inc., Boston, MA,
  1991. \MR{1124848}

\bibitem[Kad02]{Kad-2002-PNASUa}
Richard~V. Kadison, \emph{{The Pythagorean Theorem II: the infinite discrete
  case}}, Proc. Natl. Acad. Sci. USA \textbf{99} (2002), no.~8, 5217--5222.

\bibitem[Kas80]{Kas-1980-JOT}
G.~G. Kasparov, \emph{Hilbert {$C^{\ast} $}-modules: theorems of {S}tinespring
  and {V}oiculescu}, J. Operator Theory \textbf{4} (1980), no.~1, 133--150.
  \MR{587371}

\bibitem[KL17]{KL-2017-IEOT}
Victor Kaftal and Jireh Loreaux, \emph{Kadison's pythagorean theorem and
  essential codimension}, Integr. Equ. Oper. Theory \textbf{87} (2017),
  565--580.

\bibitem[KN06]{KN-2006-HJM}
Dan Kucerovsky and P.~W. Ng, \emph{The corona factorization property and
  approximate unitary equivalence}, Houston J. Math. \textbf{32} (2006), no.~2,
  531--550. \MR{2219330}

\bibitem[KNZ09]{KNZ-2009-JFA}
Victor Kaftal, Ping~Wong Ng, and Shuang Zhang, \emph{Strong sums of projections
  in von {N}eumann factors}, J. Funct. Anal. \textbf{257} (2009), no.~8,
  2497--2529. \MR{2555011 (2011b:46096)}

\bibitem[KNZ17]{KNZ-2017-CJM}
Victor {Kaftal}, Ping~Wong {Ng}, and Shuang {Zhang}, \emph{{Strict comparison
  of positive elements in multiplier algebras}}, {Can. J. Math.} \textbf{69}
  (2017), no.~2, 373--407 (English).

\bibitem[Lee11]{Lee-2011-JFA}
Hyun~Ho Lee, \emph{Proper asymptotic unitary equivalence in {$KK$}-theory and
  projection lifting from the corona algebra}, J. Funct. Anal. \textbf{260}
  (2011), no.~1, 135--145. \MR{2733573}

\bibitem[Lee13]{Lee-2013-JFA}
\bysame, \emph{Deformation of a projection in the multiplier algebra and
  projection lifting from the corona algebra of a non-simple {$C^*$}-algebra},
  J. Funct. Anal. \textbf{265} (2013), no.~6, 926--940. \MR{3067791}

\bibitem[Lee18]{Lee-2018-JMAA}
\bysame, \emph{Homotopy classification of homogeneous projections in the corona
  algebra of {$C(X,B)$} for a graph {$X$}}, J. Math. Anal. Appl. \textbf{461}
  (2018), no.~2, 1404--1415. \MR{3765498}

\bibitem[{Lin}01]{Lin-2001}
Huaxin {Lin}, \emph{{An introduction to the classification of amenable
  \(C^*\)-algebras}}, Singapore: World Scientific, 2001 (English).

\bibitem[Lin02]{Lin-2002-JOT}
Huaxin Lin, \emph{Stable approximate unitary equivalence of homomorphisms}, J.
  Operator Theory \textbf{47} (2002), no.~2, 343--378. \MR{1911851}

\bibitem[{Lin}05]{Lin-2005-CJM}
Huaxin {Lin}, \emph{{Extension by simple \(C^*\)-algebras: quasidiagonal
  extensions}}, {Can. J. Math.} \textbf{57} (2005), no.~2, 351--399 (English).

\bibitem[LN20]{LN-2020-IEOT}
Jireh Loreaux and P.~W. Ng, \emph{Remarks on essential codimension}, Integral
  Equations Oper. Theory \textbf{92} (2020), no.~4, 1--35.

\bibitem[Lor19]{Lor-2019-JOT}
Jireh Loreaux, \emph{Restricted diagonalization of finite spectrum normal
  operators and a theorem of arveson}, Journal of Operator Theory \textbf{81}
  (2019), no.~2, 257--272.

\bibitem[Ng18]{Ng-2018-NYJM}
Ping~W. Ng, \emph{A double commutant theorem for the corona algebra of a
  {R}azak algebra}, New York J. Math. \textbf{24} (2018), 157--165.
  \MR{3761942}

\bibitem[OPR11]{OPR-2011-TAMS}
Eduard {Ortega}, Francesc {Perera}, and Mikael {R{\o}rdam}, \emph{{The Corona
  factorization property and refinement monoids}}, {Trans. Am. Math. Soc.}
  \textbf{363} (2011), no.~9, 4505--4525 (English).

\bibitem[OPR12]{OPR-2012-IMRN}
\bysame, \emph{{The corona factorization property, stability, and the Cuntz
  semigroup of a \(C^*\)-algebra}}, {Int. Math. Res. Not.} \textbf{2012}
  (2012), no.~1, 34--66 (English).

\bibitem[Pas81]{Pas-1981-PJM}
William~L. Paschke, \emph{{$K$}-theory for commutants in the {C}alkin algebra},
  Pacific J. Math. \textbf{95} (1981), no.~2, 427--434. \MR{632196}

\bibitem[{Ped}90]{Ped-1990}
Gert~K. {Pedersen}, \emph{{The corona construction}}, {Operator theory, Proc.
  GPOTS-Wabash Conf., Indianapolis/IN 1988, Pitman Res. Notes Math. Ser. 225,
  49-92 (1990).}, 1990.

\bibitem[RLL00]{RLL-2000}
M.~{R{\o}rdam}, F.~{Larsen}, and N.~{Laustsen}, \emph{{An introduction to
  \(K\)-theory for \(C^*\)-algebras}}, vol.~49, Cambridge: Cambridge University
  Press, 2000 (English).

\bibitem[{R{\o}}04]{Roer-2004-IJM}
Mikael {R{\o}rdam}, \emph{{The stable and the real rank of \(\mathcal
  Z\)-absorbing \(C^*\)-algebras}}, {Int. J. Math.} \textbf{15} (2004), no.~10,
  1065--1084 (English).

\bibitem[Tho01]{Tho-2001-PAMS}
Klaus Thomsen, \emph{On absorbing extensions}, Proc. Amer. Math. Soc.
  \textbf{129} (2001), no.~5, 1409--1417. \MR{1814167}

\bibitem[Val83]{Val-1983-PJM}
Alain Valette, \emph{A remark on the {K}asparov groups {${\rm Ext}(A,\,B)$}},
  Pacific J. Math. \textbf{109} (1983), no.~1, 247--255. \MR{716300}

\bibitem[Voi76]{Voi-1976-RRMPA}
Dan Voiculescu, \emph{A non-commutative {W}eyl-von {N}eumann theorem}, Rev.
  Roumaine Math. Pures Appl. \textbf{21} (1976), no.~1, 97--113. \MR{0415338}

\bibitem[{Weg}93]{Weg-1993}
Niels~Erik {Wegge-Olsen}, \emph{\({K}\)-theory and \({C}^*\)-algebras: a
  friendly approach}, Oxford: Oxford University Press, 1993 (English).

\end{thebibliography}

\end{document}